\documentclass[AMS,STIX1COL]{WileyNJD-v2}
\newcommand\dom{\operatorname{dom}}

\DeclareMathOperator{\re}{Re}

\DeclareMathOperator{\im}{Im}
\DeclareMathOperator{\Span}{span}
\DeclareMathOperator{\ran}{ran}  
\DeclareMathOperator{\ess}{ess} 
\DeclareMathOperator{\supp}{supp} 
\DeclareMathOperator{\eq}{eq}
\DeclareMathOperator{\sech}{sech}

\newcommand{\bb}[1]{\mathbf{#1}}
\newcommand{\bs}[1]{\boldsymbol{#1}}
\newcommand{\mc}[1]{\mathcal{#1}}
\newcommand{\la}{\langle}
\newcommand{\ra}{\rangle}
\newcommand{\EE}{\mathcal E}  

\articletype{Preprint}%

\received{26 April 2016}
\revised{6 June 2016}
\accepted{6 June 2016}

\begin{document}
\title{A nonlinear Klein-Gordon equation on a star graph}

\author{Nataliia Goloshchapova}

\authormark{Nataliia Goloshchapova \textsc{}}

\address{\orgdiv{Institute of Mathematics}, \orgname{University of São Paulo}, \orgaddress{\state{São Paulo}, \country{Brazil}}}

\corres{\email{nataliia@ime.usp.br}}

\presentaddress{R. do Matão, 1010, São Paulo - SP, 05508-090}

\abstract[Summary]{We study local   well-posedness  and orbital stability/instability of standing waves for  a  first order system associated with a nonlinear Klein-Gordon equation on a star graph.  The proof of the well-posedness uses a classical fixed point argument and the Hille-Yosida theorem. Stability study  relies on the  linearization approach and recent results  for the  NLS equation with the $\delta$-interaction on a star graph.  }


\maketitle

\section{Introduction}
The study of differential  equations on  graphs is a rapidly developing area (see  \cite{BerKuc13} and the references therein). It is  motivated by various physical applications involving wave
propagation in narrow waveguides. Graphs arise as  approximations of  multi-dimensional narrow waveguides when their
thickness parameters converge to zero.
 A large part of the literature is devoted to linear equations on graphs, with special emphasis on the Schr\"odinger
equation describing the so-called quantum graphs.  The models on a  star graph $\Gamma$  constituted  by $N$ half-lines  joined at the vertex $\nu=0$ are one of the simplest.  Recently a certain amount o research  work has been done on the nonlinear Schr\"odinger
equation with the $\delta$-interaction (NLS-$\delta$)  on  $\Gamma$ (see \cite{AngGol18a, CacFin17, Noj14}, and the references therein):
 \begin{equation}\label{NLS_graph_ger}
i\partial_t \mathbf{u}(t,x)=H_\alpha\mathbf{u}(t,x)-|\mathbf{u}(t,x)|
^{p-1}\mathbf{u}(t,x),
\end{equation}
where $p>1$, $\mathbf{u}(t,x)=(u_j(t,x))_{j=1}^N:\mathbb{R}\times \mathbb{R}_+\rightarrow \mathbb{C}^N$, and $H_\alpha$ is the self-adjoint operator on $L^2(\Gamma)$ defined  by
\begin{equation}\label{D_alpha}
\begin{split}
(H_\alpha \mathbf{v})(x)&=\left(-v_j''(x)\right)_{j=1}^N,\quad x> 0,\quad  \mathbf{v}=(v_j)_{j=1}^N,\\
\dom(H_\alpha)&=\Big\{\mathbf{v}\in H^2(\Gamma): v_1(0)=\ldots=v_N(0),\,\,\sum\limits_{j=1}^N  v_j'(0)=\alpha v_1(0)\Big\}.
\end{split}
\end{equation}
The NLS-$\delta$ equation has been studied  in the
context of well-posedness, variational properties,  existence, and stability of standing waves. 
In his survey \cite{Noj14} about the NLS on graphs,  Noja along with the model \eqref{NLS_graph_ger}, mentioned  (as one of the main examples of PDEs on the star graph)   the following nonlinear Klein-Gordon equation 
with the $\delta$-interaction (NKG-$\delta$):
\begin{equation}\label{NKG_graph_ger}
-\partial_t^2 \mathbf{u}(t,x)=H_\alpha\mathbf{u}(t,x) +m^2\bb u(t,x)-|\mathbf{u}(t,x)|^{p-1}\mathbf{u}(t,x).
\end{equation}
On each edge of the graph (i.e. on each half-line) we have
$$-\partial_t^2 u_j(t,x)=-\partial_x^2u_j(t,x) +m^2u_j(t,x)-|u_j(t,x)|^{p-1}u_j(t,x),\quad x>0,\quad j\in\{1,\ldots,N\},$$ 
moreover, the vectors $\bb u(t,0)=(u_j(t,0))_{j=1}^N$ and $\bb u'(t,0)=(u'_j(t,0))_{j=1}^N$ satisfy the conditions in \eqref{D_alpha}.
To our knowledge, the   NKG-$\delta$ equation has never been studied in the context of well-posedness and stability of standing waves. 

In the present paper we  aim to initiate this  study. The stability/instability  study of standing wave solutions of the NKG equation in homogeneous media (in $n$ space dimensions) was started by Shatah in \cite{Sha83, Sha85}, and then continued in \cite{JeaCoz09, OhtTod07}. We rely on the recent research  \cite{CsoGen18}, where the authors considered the nonlinear Klein-Gordon equation with $\delta$-potentials on $\mathbb{R}$
\begin{equation*}\label{NLS_line}
-\partial_t^2 u(t,x)=-\partial_x^2u(t,x) +m^2u(t,x)+\gamma\delta(x)u(t,x)+i\alpha\delta(x)\partial_tu(t,x)-|u(t,x)|^{p-1}u(t,x),
\end{equation*}  
$\gamma,\alpha\in \mathbb{R}$, and $\delta(x)$ is Dirac delta function.

We prove local well-posedness  of the  Cauchy problem for a first order Hamiltonian system associated with \eqref{NKG_graph_ger}, using  classical approach related on the theory of $C_0$  semigroups (see \cite{CazHar98, Paz83} for the detailed exposition).  In particular, it has been shown that certain operator $A-\beta$ associated with equation \eqref{NKG_graph_ger} is dissipative. 

The main goal is the study of orbital stability of the standing wave solutions $\bb u(t,x)= e^{i\omega t}\bs\varphi(x)$ to \eqref{NKG_graph_ger}, where the profile $\bs\varphi(x)$ is a real-valued vector function. The profile $\bs \varphi(x)$ satisfies the following stationary equation \begin{equation}\label{H_alpha}
H_\alpha\bs{\varphi}+(m^2-\omega^2)\bs{\varphi}-|\bs{\varphi}|
^{p-1}\bs{\varphi}=0.
\end{equation}
Applying  \cite[Theorem 4]{AdaNoj14},  one gets  the description of  real-valued vector  solutions to \eqref{H_alpha}.
\begin{theorem}
Let  $[s]$ denote the integer part of $s\in\mathbb{R}$, and $\alpha\neq 0$.  Then equation \eqref{H_alpha} has $\left[\tfrac{N-1}{2}\right]+1$ (up to permutations of the edges of $\Gamma$) real-valued vector solutions $\bs{\varphi}_{k,\omega}^\alpha=(\tilde{\varphi}_{k,j})_{j=1}^N, \,\,k\in\left\{0,\ldots,\left[\tfrac{N-1}{2}\right]\right\}$, which are given by
\begin{equation}\label{Phi_k}
\begin{split}
 \tilde{\varphi}_{k,j}(x)&= \left\{
                    \begin{array}{ll}
                      \Big[\frac{(p+1)(m^2-\omega^2)}{2} \sech^2\Big(\frac{(p-1)\sqrt{m^2-\omega^2}}{2}x-c_k\Big)\Big]^{\frac{1}{p-1}}, & \quad\hbox{$j=1,\ldots,k$;} \\
                     \Big[\frac{(p+1)(m^2-\omega^2)}{2} \sech^2\Big(\frac{(p-1)\sqrt{(m^2-\omega^2)}}{2}x+c_k\Big)\Big]^{\frac{1}{p-1}}, & \quad\hbox{$j=k+1,\ldots,N,$}
                    \end{array}
                  \right.\\
                  \text{where}\;\; c_k&=\tanh^{-1}\bigg(\tfrac{\alpha}{(2k-N)\sqrt{(m^2-\omega^2)}}\bigg),\,\,\text{and}\,\,\,\,m^2-\omega^2>\tfrac{\alpha^2}{(N-2k)^2}.
                  \end{split}
\end{equation}
\end{theorem}
In Theorem \ref{main} we provide a sufficient condition on the parameters $\omega, m, \alpha, k, N$ to get the orbital stability/instability of the standing waves $e^{i\omega t}\bs{\varphi}_{k,\omega}^\alpha(x)$. 
The orbital stability is studied in the context of a Hamiltonian system  associated with the NKG-$\delta$ equation.  Its investigation   relies  on the classical works by Grillakis, Shatah, and Strauss \cite{GrilSha87, GrilSha90} and recent work  \cite{ Stu08} by Stuart.
The proof of stability/instability  result essentially uses spectral analysis of certain self-adjoint Schr\"odinger operators on the star graph. This analysis was elaborated extensively in  papers \cite{AngGol18, AngGol18a}   devoted to the stability study of standing waves for the  NLS-$\delta$  equation. The principal ingredients of the spectral analysis are the analytic perturbation theory and the  extension theory of symmetric operators.
\vspace{0.5 cm}

\noindent{\bf Notation.}

Let $L$ be  a densely defined symmetric operator in some Hilbert space.
\,The deficiency  numbers of $L$ are defined by   $n_\pm(L):=\dim \ker(L^*\mp iI)$. The number of negative eigenvalues counting multiplicities (\textit{the Morse index}) is denoted by  $n(L)$. 

We regard  $L^2(\mathbb{R}_+)$  as a real Hilbert space with the inner product 
$$\la u, v\ra_{L^2(\mathbb{R}_+)}=\re\int\limits_{\mathbb{R}_+}u\overline{v}dx,$$
and  $H^1(\mathbb{R}_+)$ as the Sobolev space with the inner product
 $$\la u, v\ra_{H^1(\mathbb{R}_+)}=\la u, v\ra_{L^2(\mathbb{R}_+)}+\la u', v'\ra_{L^2(\mathbb{R}_+)}.$$

  We consider   the star graph $\Gamma$  constituted by $N$ half-lines $\mathbb{R_+}$ attached  to a common vertex $\nu=0$.  The function $\bb w$ acting on $\Gamma$ is represented by the vector $(w_j)_{j=1}^N,$ where each scalar function $w_j$ is defined on $[0,\infty)$. For   $\bb w=(w_j)_{j=1}^N$  on $\Gamma$, we will abbreviate
$$\int\limits_\Gamma\bb w dx=\sum\limits_{j=1}^N\int\limits_{\mathbb{R}_+}w_j dx.$$
On the graph we define the following spaces 
\begin{equation*}
 L^q(\Gamma)=\bigoplus\limits_{j=1}^NL^q(\mathbb{R}_+),\, 1\leq q\leq \infty,\quad H^1(\Gamma)=\bigoplus\limits_{j=1}^NH^1(\mathbb{R}_+),\quad H^2(\Gamma)=\bigoplus\limits_{j=1}^NH^2(\mathbb{R}_+). 
\end{equation*}  
The corresponding $L^2$- and $H^1$-inner products are defined by  
$$\la \bb u, \bb v\ra_{L^2(\Gamma)}=\re\int\limits_\Gamma\bb u\overline{\bb v}dx, \qquad  \la \bb u, \bb v\ra_{H^1(\Gamma)}=\re\left[\int\limits_\Gamma\bb u\overline{\bb v}dx+\int\limits_\Gamma\bb u'\overline{\bb v}'dx\right].$$ 
By $\EE (\Gamma)$ we denote the space
$$\mathcal{E}(\Gamma)
=\{\bb v\in H^1(\Gamma): \,  v_1(0)=\ldots=v_N(0)\}.$$
The dual space for $\mc E(\Gamma)$ is denoted by $\mc E^*(\Gamma)$.

Set $X=\mc E(\Gamma)\times L^2(\Gamma)=\{(\bb u,\bb v): \bb u\in \mc E(\Gamma), \bb v\in  L^2(\Gamma)\}$ for the real Hilbert space  with the inner product
$$\la(\bb u_1, \bb v_1), (\bb u_2, \bb v_2)\ra_X=\la\bb u_1, \bb u_2\ra_{H^1(\Gamma)}+\la\bb v_1, \bb v_2\ra_{L^2(\Gamma)}.$$
Its dual  $X^*$  is identified  with $\mc E^*(\Gamma)\times L^2(\Gamma)$, and the duality pairing is denoted by  
$\la\cdot, \cdot\ra_{X^*\times X}.$ For $k\in\{0,\ldots, N-1\}$ we define  the spaces 
\begin{equation*}
\begin{split}
  L^2_k(\Gamma)&=\{\bb v\in L^2(\Gamma): v_1(x)=\ldots=v_k(x),\,v_{k+1}(x)=\ldots=v_N(x), x\geq 0\},
 \qquad \mbox{and}\\  
 \mc E_k(\Gamma)&=\mc E(\Gamma)\cap L^2_k(\Gamma),\quad X_k=\mc E_k(\Gamma)\times L^2_k(\Gamma).
\end{split}
\end{equation*}
If $k=0$, then  $L^2_{\eq}(\Gamma)=L^2_0(\Gamma)$, $\mc E_{\eq}(\Gamma)=\mc E(\Gamma)\cap L^2_{\eq}(\Gamma)$, and $X_{\eq}=\mc E_{\eq}(\Gamma)\times L^2_{\eq}(\Gamma).$

\section{Local well-posedness}\label{sec2}
We consider the following Cauchy problem
\begin{align}\label{Cauchy1}
\left\{\begin{array}{c}
-\partial_t^2 \mathbf{u}(t,x)=H_\alpha\mathbf{u}(t,x) +m^2\bb u(t,x)-|\mathbf{u}|^{p-1}\mathbf{u},\\
\bb u(0,x)=\bb u_0(x),\qquad\qquad\qquad\qquad\qquad\qquad\\
\bb \partial_t\bb u(0,x)=\bb u_1(x).\qquad\qquad\qquad\qquad\qquad\qquad
\end{array}\right.
\end{align}
Let us reformulate \eqref{Cauchy1} as a first order system on $X$. Denoting $\bb v=\partial_t\bb u$, $\bb U=(\bb u, \bb v)$, $F(\bb U)=(0, |\bb u|^{p-1}\bb u)$,  and $\bb U_0=(\bb u_0, \bb u_1),$ we formally get from \eqref{Cauchy1}
\begin{align}\label{Cauchy2}
\left\{\begin{array}{c}
\partial_t\bb U(t)=A\bb U(t)+F(\bb U(t)),\\
\bb U(0)=\bb U_0,\qquad\qquad\qquad
\end{array}\right.
\end{align}
 where 
 \begin{equation*}
 \begin{split}
 &A=\left(\begin{array}{cc}0&Id_{L^2(\Gamma)}\\
 -H_\alpha-m^2& 0 
 \end{array}\right),\\
 \dom(A)=&\left\{\begin{array}{c}
 (\bb u, \bb v)\in H^2(\Gamma)\times \mc E(\Gamma): \\ 
 u_1(0)=\ldots=u_N(0),\,\, \sum\limits_{j=1}^N  u_j'(0)=\alpha u_1(0)\end{array}\right\}=\dom(H_\alpha)\times \mc E(\Gamma). 
 \end{split}
  \end{equation*}
   Below we will prove  existence and uniqueness of a weak solution  $\bb U(t)\in C([0,T], X)$ to system \eqref{Cauchy2} (see \cite{Bal77} for the definition of a weak solution).  The proof is in the spirit of \cite[Chapter 4]{CazHar98}.
   First, we  prove that operator $A$ generates strongly continuous  semigroup on $X$.
   \begin{proposition}\label{C_0}
 The operator $A$ generates $C_0$-semigroup on $X$. Moreover, there exist $M\geq 0$ and $\beta\geq 0$ such that for all $t\geq 0$ the following estimate holds
 \begin{equation}\label{est_group}
\| e^{tA}\|\leq M e^{\beta t}.
 \end{equation}
   \end{proposition}
\begin{proof}
Our aim is to apply \cite[Chapter I, Corollary 3.8]{Paz83}. We need to prove density of $\dom(A)$ in $X$.

\textit{Step 1.} Let $(\bb u, \bb v)\in X$. Obviously  there exists a sequence $\{\bb v_n\}_{n=1}^\infty\subset  \mc E(\Gamma)$
such that $\bb v_n\underset{n\to \infty}{\longrightarrow} \bb v$ in $L^2(\Gamma)$ (indeed, $\dom(H_\alpha)\subset \mc E(\Gamma)$ and  $\overline{\dom(H_\alpha)}=L^2(\Gamma)$).  We  need to show that  there exists a sequence  $\{\bb u_n\}_{n=1}^\infty\subset \dom(H_\alpha)$ such that $\bb u_n\underset{n\to \infty}{\longrightarrow} \bb u\in \mc E(\Gamma)$ in $H^1(\Gamma)$.
Consider the following self-adjoint operator $H_0$ in $L^2(\Gamma)$ (with the Kirchhoff conditions)
\begin{equation}\label{Kirchhoff}
\begin{split}
(H_0 \mathbf{w})(x)&=\left(-w_j''(x)\right)_{j=1}^N,\quad x> 0,\quad  \mathbf{w}=(w_j)_{j=1}^N,\\
\dom(H_0)&=\Big\{\mathbf{w}\in H^2(\Gamma): w_1(0)=\ldots=w_N(0),\,\,\sum\limits_{j=1}^N  w_j'(0)=0\Big\}.
\end{split}
\end{equation}
We show that there exists a sequence  $\{\tilde{\bb u}_n\}_{n=1}^\infty\subset \dom(H_0)$ such that $\tilde{\bb u}_n\underset{n\to \infty}{\longrightarrow} \bb u$ in $H^1(\Gamma)$, that is, $\overline{\dom(H_0)}=\mc E(\Gamma).$ It is sufficient to show  that orthogonal complement of $\dom(H_0)$ in $\mc E(\Gamma)$ is $\{0\}$.
 
 Let $\bb z\in \dom(H_0)^\perp$ in $\mc E(\Gamma)$, hence for any $\bb w\in \dom(H_0)$
 \begin{align*}
 &\la\bb w, \bb z\ra_{H^1(\Gamma)}=\la\bb w, \bb z\ra_{L^2(\Gamma)}+\la\bb w', \bb z'\ra_{L^2(\Gamma)}\\&=\la\bb w, \bb z\ra_{L^2(\Gamma)}-\re\Big(\sum\limits_{j=1}^N w'_j(0)\overline z_j(0)\Big)-\la\bb w'',\bb z\ra_{L^2(\Gamma)}=\la-\bb w''+\bb w, \bb z\ra_{L^2(\Gamma)}=0.
 \end{align*}
 The last equality implies  $\bb z\in \ran(H_0+1)^\perp=\ker(H_0+1)=\{0\}$. 
  
  We modify the sequence $\{\tilde{\bb u}_n \}_{n=1}^\infty$ to get another one $\{\bb u_n \}_{n=1}^\infty\subset\dom(H_\alpha)$ that approximates $\bb u$ in $H^1(\Gamma)$. 
  Define the sequence $\{\bs\zeta_n\}_{n=1}^\infty$ by $\bs\zeta_n=(1+\frac{\alpha x}{N}e^{-nx^2})_{j=1}^N$. Let us show that $\{\bb u_n\}_{n=1}^\infty:=\{\bs\zeta_n\tilde{\bb u}_n\}_{n=1}^\infty\subset \dom(H_\alpha)$ and $\bb u_n\underset{n\to \infty}{\longrightarrow} \bb u$ in $H^1(\Gamma).$  
 It is easily seen that 
 $$ u_{jn}(0)=\tilde{u}_{jn}(0),\quad u'_{jn}(0)=\frac{\alpha}{N}\tilde{u}_{jn}(0)+\tilde{u}'_{jn}(0),$$
which yields 
$$\sum\limits_{j=1}^Nu'_{jn}(0)=\sum\limits_{j=1}^N\left(\frac{\alpha}{N}\tilde{u}_{jn}(0)+\tilde{u}'_{jn}(0)\right)=\alpha \tilde{u}_{1n}(0)= \alpha u_{1n}(0).$$
Therefore, $\{\bb u_n\}_{n=1}^\infty\subset\dom(H_\alpha).$
Observing that 
$$\|\bs\zeta_n-\bb 1\|_{L^\infty(\Gamma)}\underset{n\to \infty}{\longrightarrow}0,\qquad \|\bs\zeta'_n\|_{L^\infty(\Gamma)}\underset{n\to \infty}{\longrightarrow}0, $$
we conclude  $\|\bb u_n-\tilde{\bb u}_n\|_{H^1(\Gamma)}\underset{n\to \infty}{\longrightarrow}0$, consequently $\|\bb u_n-\bb u\|_{H^1(\Gamma)}\underset{n\to \infty}{\longrightarrow}0$. Finally, $\|(\bb u_n, \bb v_n)-(\bb u,\bb v)\|_X\underset{n\to \infty}{\longrightarrow}0$.
 
\textit{Step 2.} 
 To prove inequality  $(3.21)$ on the  resolvent $(A-\lambda)^{-1}$ in \cite[Chapter I]{Paz83}, we introduce alternative equivalent  norm on $X$.
It is known that $\inf \sigma(H_\alpha)=\left\{\begin{array}{c}
0,\,\, \alpha\geq 0,\\
-\frac{\alpha^2}{N^2},\,\, \alpha<0.
\end{array}\right.$ See, for example,   Proposition \ref{spect_L_j}  below for the proof of the identity $\sigma_{\ess}(H_\alpha)=[0,\infty).$  The analysis of the discrete spectrum is trivial. 

Given $\mu^2>-\inf \sigma(H_\alpha)$, then  denoting $$\|\bb u\|^2_{1,\mu}:=\|\bb u'\|^2_{L^2(\Gamma)}+\mu^2\|\bb u\|^2_{L^2(\Gamma)}+\alpha|u_1(0)|^2,$$ by the Sobolev embedding and inequality \begin{equation}\label{sigma_inf}
\|\bb u'\|^2_{L^2(\Gamma)}+\alpha|u_1(0)|^2\geq \inf \sigma(H_\alpha)\|\bb u\|_{L^2(\Gamma)},\end{equation}  we get
$$C_1\|\bb u\|_{H^1(\Gamma)}\leq \|\bb u\|_{1,\mu}\leq C_2\|\bb u\|_{H^1(\Gamma)}.$$
Therefore, the quadratic form defined on $X$ by 
$$\|(\bb u,\bb v)\|^2_{X,\mu}=\|\bb u\|^2_{1,\mu}+\|\bb v\|^2_{L^2(\Gamma)}$$
gives a norm on $X$ equivalent to $\|\cdot\|_X$.
 The corresponding inner product is given by 
 $$\la(\bb u_1, \bb v_1), (\bb u_2, \bb v_2)\ra_{X,\mu}=\la\bb u'_1, \bb u'_2\ra_{L^2(\Gamma)}+\mu^2\la\bb u_1, \bb u_2\ra_{L^2(\Gamma)}+\alpha\re(u_{11}(0)\overline{u}_{12}(0))+\la\bb v_1, \bb v_2\ra_{L^2(\Gamma)}.$$

\textit{Step 3.} 
Suppose that $\mu$ is such that $\mu^2> m^2$. 
Let  $\bb U=(\bb u,\bb v)\in \dom(A)$, then
\begin{align*}
&\la A\bb U,\bb U\ra_{X,\mu}=\la\bb v', \bb u'\ra_{L^2(\Gamma)}+\mu^2\la\bb v, \bb u\ra_{L^2(\Gamma)}+\alpha\re(v_1(0)\overline{u}_1(0))+\la\bb u'', \bb v\ra_{L^2(\Gamma)}-m^2\la\bb u, \bb v\ra_{L^2(\Gamma)}\\&=\la\bb v', \bb u'\ra_{L^2(\Gamma)}+(\mu^2-m^2)\la\bb v, \bb u\ra_{L^2(\Gamma)}+\alpha\re(v_1(0)\overline{u}_1(0))-\la\bb u', \bb v'\ra_{L^2(\Gamma)}-\alpha\re(u_1(0)\overline{v}_1(0))\\&=(\mu^2-m^2)\la\bb v, \bb u\ra_{L^2(\Gamma)}.
\end{align*}
Hence $$|\la A\bb U,\bb U\ra_{X,\mu}|=(\mu^2-m^2)|\la\bb v, \bb u\ra_{L^2(\Gamma)}|\leq \frac{\mu^2-m^2}{2}(\|\bb u\|^2_{L^2(\Gamma)}+\|\bb v\|^2_{L^2(\Gamma)}).$$
Observing that 
$$\|\bb U\|^2_{X,\mu}=\|\bb u\|^2_{1,\mu}+\|\bb v\|^2_{L^2(\Gamma)}\geq C^2_1\|\bb u\|^2_{H^1(\Gamma)}+\|\bb v\|^2_{L^2(\Gamma)}\geq C(\|\bb u\|^2_{L^2(\Gamma)}+\|\bb v\|^2_{L^2(\Gamma)}),$$
for $\beta\geq 0$ large enough one gets
$$\la (A-\beta)\bb U,\bb U\ra_{X,\mu}\leq 0.$$
Therefore, by \cite[Proposition 2.4.2]{CazHar98}, the operator $ A-\beta$ is dissipative. By dissipativity, for $\lambda>\beta$ one easily gets
\begin{equation}\label{eq1}
\begin{split}
\|(A-\lambda)\bb U\|^2_{X,\mu}=2(\beta-\lambda)\la(A-\beta)\bb U,\bb U\ra_{X,\mu}&+\|(A-\beta)\bb U\|^2_{X,\mu}+(\lambda-\beta)^2\|\bb U\|^2_{X,\mu}\\&\geq \|(A-\beta)\bb U\|^2_{X,\mu}+(\lambda-\beta)^2\|\bb U\|^2_{X,\mu}.\end{split}\end{equation}
The above inequality implies that $\ker(A-\lambda)=\{0\}$. We show that $ A-\lambda$ is surjective, i.e. $\ran(A-\lambda)= X$.  Let $\bb F=(\bb f,\bb g)\in X.$ We prove that there exists $\bb U=(\bb u,\bb v)\in \dom(A)$ such that $(A-\lambda)\bb U=\bb F$ or equivalently $\left\{\begin{array}{c}
\bb v=\lambda \bb u+\bb f,\qquad\qquad\qquad\quad\\
-H_\alpha\bb u-(m^2+\lambda^2)\bb u=\bb g+\lambda\bb f.
\end{array}\right.
$

It is obvious that for $\lambda>\beta>\sqrt{\mu^2-m^2}$ the  equation  $-H_\alpha\bb u-(m^2+\lambda^2)\bb u=\bb g+\lambda\bb f$ has a unique solution $\bb u=-(H_\alpha+m^2+\lambda^2)^{-1}(\bb g+\lambda\bb f)$ (indeed, $-\lambda^2-m^2\in\rho(H_\alpha)$), and therefore $\bb v=-\lambda(H_\alpha+m^2+\lambda^2)^{-1}(\bb g+\lambda\bb f)+\bb f$. This implies $\ran(A-\lambda)=X,$ and finally, by estimate \eqref{eq1}, the operator   $A-\lambda$ has a bounded everywhere defined inverse, i.e.
$$\|(A-\lambda)^{-1}\|\leq \frac{1}{\lambda-\beta}.$$ In particular, for $\beta$ large enough $(\beta, \infty)\subset\rho (A)$,   and hence $A$ is closed.
Recalling that $\overline{\dom(A)}=X$, by corollary of the Hille-Yosida theorem (see \cite[Chapter I, Corollary 3.8]{Paz83}), $A$ generates a $C_0$ semigroup on $(X,\|\cdot\|_{X,\mu})$. Moreover, for   $t\geq 0$ and $\bb U\in X$, we have 
$$\|e^{tA}\bb U\|_{X,\mu}\leq e^{\beta t}\|\bb U\|_{X,\mu}.$$ By equivalence of the norms, there exists $M>0$ such that 
$$\|e^{tA}\bb U\|_{X}\leq M e^{\beta t}\|\bb U\|_{X}.$$ 
\end{proof}  
 \begin{remark}
 Observe that analogously one might show that for $\beta$ large enough $(\beta, \infty)\subset\rho(-A)$ and $\|(-A-\lambda)^{-1}\|\leq \frac{1}{\lambda-\beta}$, and therefore 
 $$\|e^{-tA}\|\leq M e^{\beta t},\quad t\geq 0,\quad\Longleftrightarrow \quad \|e^{tA}\|\leq M e^{-\beta t},\quad t\leq 0.$$
 Hence $$\|e^{tA}\|\leq M e^{\beta |t|},\quad t\in\mathbb{R}.$$ 
 \end{remark}
 \begin{remark}
 The fact that $\overline{\dom(H_\alpha)}=\mc E(\Gamma)$ can be shown alternatively applying  the Representation Theorem \cite[Chapter VI, Theorem 2.1]{Kat66}. Namely, one can use that $\dom(H_\alpha)$ is a core of the form $$t_\alpha(\bb u, \bb v)=\la\bb u', \bb v'\ra_{L^2(\Gamma)}+\alpha\re(u_1(0) \overline{v_1}(0)), \quad \dom(t_\alpha)=\mc E(\Gamma).$$ 
 \end{remark}

Using Proposition \ref{C_0}, we obtain the following well-posedness theorem.

\begin{theorem}\label{well_X}
Let $p > 1$. Then 
\begin{itemize}
\item[$(i)$] for any $\bb U_0=(\bb u_0,\bb u_1)\in X$ there exists $T >
0$ such that problem \eqref{Cauchy2} has a unique weak solution $\bb U(t)=(\bb u(t),\bb v(t))\in C([0, T], X)$;
\item[$(ii)$]  problem \eqref{Cauchy2} has a maximal solution defined on
an interval of the form  $[0, T_{\max})$, and the following blow-up alternative holds: either $T_{\max} = \infty$ or $T_{\max}<\infty$ and
$$\lim\limits_{t\to T_{\max}}\|\bb U(t)\|_{X} =\infty;$$
\item[$(iii)$]
 for each $T\in (0, T_{\max})$ the mapping
$\bb U_0\in X\mapsto \bb U(t)\in C([0, T], X)$ is continuous;
\item[$(iv)$] the solution $\bb U(t)$ satisfies conservation of charge and energy:
\begin{equation*}\label{conservation-laws}
 E(\bb U(t))= E(\bb U_0),\quad Q(\bb U(t))=Q(\bb U_0)\,\,\,\text{for all}\,\,\,t\in[0, T_{\max}),
\end{equation*}
 where 
\begin{equation*}\label{en_ch}
\begin{split}
 &E(\bb u, \bb v)
=\frac{1}{2}\|\bb u'\|_{L^2(\Gamma)}^2+\frac{\alpha}{2}|u_1(0)|^2+\frac{m^2}{2}\|\bb u\|_{L^2(\Gamma)}^2-\frac{1}{p+1}\|\bb u\|_{L^{p+1}(\Gamma)}^{p+1}+\frac{1}{2}\|\bb v\|_{L^2(\Gamma)}^2,\\
& Q(\bb u, \bb v)=\im\int\limits_{\Gamma}\bb u\overline{\bb v}dx.
\end{split}
\end{equation*}
\end{itemize}
\end{theorem}
\begin{proof}
$(i)$\,  Firstly, we prove that the nonlinearity $F$ is Lipschitz continuous on bounded sets of $X$, i.e.  
\begin{equation}\label{lip}\|F(\bb U)-F(\bb W)\|_X\leq C(R)\|\bb U-\bb W\|_X,\end{equation}
for  $\bb U, \bb W\in X$, $R>0$ with $\|\bb U\|_X\leq R, \|\bb W\|_X\leq R$.

For $\bb u=(u_j)_{j=1}^N$ and $\bb w=(w_j)_{j=1}^N$ one has 
$$\|u_j|^{p-1}u_j-|w_j|^{p-1}w_j|\leq C(|u_j|^{p-1}+|w_j|^{p-1})|u_j-w_j|,$$
which implies 
\begin{equation*}
\||\bb u|^{p-1}\bb u-|\bb w|^{p-1}\bb w\bb \|_{L^2(\Gamma)}\leq C_1(\|\bb u\|_{L^\infty(\Gamma)}^{p-1}+\|\bb w\|_{L^\infty(\Gamma)}^{p-1})\|\bb u-\bb w\|_{L^2(\Gamma)}.
\end{equation*}
 Therefore, by the  Gagliardo-Nirenberg inequality  (see formula (2.3) in \cite{AdaNoj14}),\begin{equation*}\label{G-N}
\|\bb\Psi\|_{L^p(\Gamma)}\leq C\|\bb\Psi'\|_{L^2(\Gamma)}^{\frac{1}{2}-\frac{1}{p}}\|\bb\Psi\|_{L^2(\Gamma)}^{\frac{1}{2}+\frac{1}{p}},\,\, p\in[2,\infty],\,\, \bb\Psi\in H^1(\Gamma),
\end{equation*}  we have
\begin{equation*}
\||\bb u|^{p-1}\bb u-|\bb w|^{p-1}\bb w\|_{L^2(\Gamma)}\leq C(R)\|\bb u-\bb w\|_{L^2(\Gamma)}.
\end{equation*} 
 Hence for $\bb U=(\bb u,\bb v), \bb W=(\bb w, \bb z)\in X$
\begin{align*}
\|F(\bb U)-F(\bb W)\|_X=\||\bb u|^{p-1}\bb u-|\bb w|^{p-1}\bb w\|_{L^2(\Gamma)}\leq C(R)\|\bb u-\bb w\|_{L^2(\Gamma)}\leq C(R)\|\bb U-\bb W\|_X. 
\end{align*}

 Secondly, we  show the existence of the solution.  Now let $R$ and $T$ be two positive constants to be defined later.
Consider the set
$$X_R:=\{\bb U(t)\in C([0,T], X): \|\bb U(t)\|_X\leq R\},$$
and the metric
$$d(\bb U,\bb V):=\max\limits_{t\in [0,T]}\|\bb U(t)-\bb V(t)\|_X.$$
Observe that $(X_R,d)$ is a complete metric space.
Now we define the map
$$\mc H(\bb U)(t)=e^{tA}\bb U_0+\int\limits_0^te^{(t-s)A}F(\bb U(s))ds.$$
It is obvious  that $\mc H:X_R\mapsto C([0, T], X)$.
We choose $T$ in order to guarantee invariance of $X_R$ for the mapping $\mc H$. By \eqref{est_group} and \eqref{lip}, we get
\begin{equation*}
\|\mc H(\bb U)(t)\|_X\leq \|e^{tA}\bb U_0\|_X+\int\limits_0^t\|e^{(t-s)A}F(\bb U(s))\|_Xds\leq M e^{\beta T}\|\bb U_0\|_{X}+TC(R)M e^{\beta T}R.
\end{equation*}
Let $\frac R{4}= M\|\bb U_0\|_X$. By choosing $T$ small enough (for example, take  $T\leq \min\left\{\frac{\ln 2}{\beta}, \frac{1}{4C(R)M}\right\}$), we get 
$$\|\mc H(\bb U)(t)\|_X\leq R.$$ 
And finally, we need to choose $T$ to get the contraction property of $\mc H$.
For $\bb U,\bb V\in X_R$ one has
\begin{equation*}
\begin{split}
\|\mc H(\bb U)(t)-\mc H(\bb V)(t)\|_X\leq \int\limits_0^t \|e^{(t-s)A}\Big(F(\bb U(s))-F(\bb V(s))\Big)\|_X ds\leq M e^{T\beta}C(R)T d(\bb U,\bb V).
\end{split}
\end{equation*}
It is easily seen that $T$ can be chosen small enough to satisfy $M e^{T\beta}C(R)T< 1$. Thus, the existence of the solution is established by the Banach fixed point theorem.

 Thirdly, we prove the uniqueness of the solution. It follows from Gronwall's lemma. Indeed, suppose that $\bb U_1,\bb U_2\in C([0,T], X)$ are two solutions and $\|\bb U_j(t)\|_X\leq K,\,\, j\in\{1,2\}, t\in[0,T].$ Then
\begin{align*}
\|\bb U_1(t)-\bb U_2(t)\|_X\leq \int\limits_0^t\|e^{(t-s)A}\Big(F(\bb U_1(s))-F(\bb U_2(s))\Big)\|_Xds\leq  M e^{T\beta}C(K)\int\limits_0^t\|\bb U_1(s)-\bb U_2(s)\|_Xds,
\end{align*} hence $\bb U_1(t)=\bb U_2(t),\, t\in[0,T].$ 

$(ii)$\,\,The blow-up alternative  follows  by a bootstrap argument (see \cite[Theorem 4.3.4]{CazHar98}). 

$(iii)$\, Repeating the proof of \cite[Proposition 4.3.7]{CazHar98}, we can show lower semicontinuity of $T:X\to (0, \infty]$  and continuous dependence: if $\bb U^n_0\underset{n\to\infty}{\longrightarrow } \bb U_0$ in $X$ and  $T < T_{\max}$, then $\bb U^n(t)\underset{n\to\infty}{\longrightarrow } \bb U(t)$ in $C([0, T], X)$.

$(iv)$\, Finally, we  show the conservation laws. 
Firstly, observe that, using \cite[Corollary 1.4.41]{CazHar98} and \cite[Proposition 4.1.6]{CazHar98}, one can prove the regularity property: for $\bb U_0\in \dom (A)$,  there exists $T >
0$ such that  problem \eqref{Cauchy2} has a unique solution $\bb U(t)\in C([0, T], \dom(A))\cap C^1([0,T],X)$ (see also \cite[Proposition 4.3.9]{CazHar98}).
Secondly, let us prove that the conservation of charge and energy hold for the solution $\bb U(t)=(\bb u(t), \partial_t\bb u(t))$ with $\bb U_0\in\dom(A).$ Using the regularity property, one  shows that
\begin{align*}\label{deriv_conserv}
&\dfrac{d}{dt}Q(\bb U(t))=\im\int\limits_\Gamma\bb u\overline{\partial^2_{t}\bb u}dx, \\
&\dfrac{d}{dt}E(\bb U(t))=\la\partial_t\bb u,\partial^2_{t}\bb u\ra_{L^2(\Gamma)}+\la H_\alpha\bb u+m^2\bb u-|\bb u|^{p-1}\bb u, \partial_t\bb u\ra_{L^2(\Gamma)}.
\end{align*}  
From \eqref{NKG_graph_ger} we get
$$\dfrac{d}{dt}Q(\bb U(t))=\im\int\limits_\Gamma\bb u\overline{\partial^2_{t}\bb u}dx=\im\left[-\int\limits_\Gamma\bb u\overline{H_\alpha \bb u}dx-m^2\int\limits_\Gamma\bb u\overline{\bb u}dx+\int\limits_\Gamma\bb u\overline{g(\bb u)}dx\right]=0,$$
hence the charge is conserved.

Multiplying \eqref{Cauchy2} by $(\overline {\partial^2_t \bb u}, -\overline{\partial_t\bb u})$ and  integrating over $\Gamma$, we obtain
$$0=\la\partial_t\bb u, {\partial^2_t\bb u}\ra_{L^2(\Gamma)}-\la{\partial^2_t\bb u}, \partial_t\bb u\ra_{L^2(\Gamma)}=\la\partial_t\bb u, {\partial^2_t\bb u}\ra_{L^2(\Gamma)}+\la H_\alpha\bb u+m^2\bb u-|\bb u|^{p-1}\bb u, \partial_t\bb u\ra_{L^2(\Gamma)},$$
therefore, the energy is conserved.

Consider now $\bb U_0\in X$, then there exists a unique solution  $\bb U(t)\in C([0, T], X)$.  Take  $\{\bb U_0^n\}_{n=1}^\infty\subset\dom(A)$ such that $\bb U^n_0\underset{n\to\infty}{\longrightarrow } \bb U_0$ in $X$. By the  regularity property,  $$\bb U^n(t)\in C([0, T^n], \dom(A))\cap C^1([0,T^n],X), \quad T^n<T^n_{\max}.$$ For each $\bb U^n(t)$ the conservation laws hold:
\begin{equation}\label{U^n}
E(\bb U^n(t))= E(\bb U^n_0),\quad Q(\bb U^n(t))=Q(\bb U^n_0), \,\,\, t\in[0, T^n_{\max}).\end{equation}
 By continuous dependence and lower semicontinuity of $T$ we have that $\bb U^n(t)\underset{n\to\infty}{\longrightarrow } \bb U(t)$ in $X$  for any  $0\leq t\leq T< T^n_{\max}$ (as $n$ is sufficiently large). 
Passing to the limit in \eqref{U^n}, we obtain the result.

\end{proof}
\begin{remark}
It is interesting to note that the conservation laws might be alternatively proved using  \cite[Theorem 6.8]{BieGen15}. We need to show that the triple $(X,\dom(A), \mc J),$ where 
$$\mc J: X\rightarrow X^*,\quad \mc J(\bb u, \bb v)=(-\bb v, \bb u),$$
 is a  symplectic Banach triple (see \cite[Definition 6.5]{BieGen15}). It is easily seen that  $\mc J $ is a symplector. 
 In order to apply   \cite[Theorem 6.8]{BieGen15} we need to  prove that  $E,Q\in \mathrm{Dif}(\dom(A),\mc J)$. It means that  $E$ and $Q$ have to be differentiable on $\dom(A)$ and $E'(\bb u, \bb v), Q'(\bb u, \bb v)$  have to belong to $\ran(\mc J)$ for any $(\bb u, \bb v)\in \dom(A)$. 
A simple check  shows that for $(\bb u, \bb v)\in\dom(A)$ one gets 
 \begin{equation}\label{func_deriv}
 E'(\bb u, \bb v)=(-\bb u''+m^2\bb u-|\bb u|^{p-1}\bb u,\bb v),\quad Q'(\bb u, \bb v)=(i\bb v,-i\bb u),
 \end{equation}
and obviously   $E'(\bb u, \bb v), Q'(\bb u, \bb v)\in \ran(\mc J).$

To conclude the proof of the conservation laws we need to observe that 
$\{E,E\}(\bb u, \bb v)=\{Q, E\}(\bb u, \bb v)=0$ for all $(\bb u, \bb v)\in\dom(A)$, where $\{\cdot,\cdot\}$ is the Poisson bracket defined by
$$\{E,E\}(\bb u, \bb v)=\la E'(\bb u, \bb v),\mc J^{-1}E'(\bb u, \bb v)\ra_{X^*\times X}, \quad \{Q, E\}(\bb u, \bb v)=\la Q'(\bb u, \bb v),\mc J^{-1}E'(\bb u, \bb v)\ra_{X^*\times X}. $$
\end{remark}
\begin{remark}
Using definition of $\mc J$ and $E$, we can reformulate the system $\partial_t\bb U(t)=A\bb U(t)+F(\bb U(t))$ in the Hamiltonian form
\begin{equation}\label{Hamil}
\mc J\partial_t\bb U(t)=E'(\bb U(t)).
\end{equation}
\end{remark}
We finish this  section by proving that  problem \eqref{Cauchy2} is well-posed in $X_k$.
\begin{lemma}\label{well_X_k}
Let $k\in \{0,\ldots, N-1\}$. For any $\bb U_0=(\bb u_0,\bb u_1)\in X_k$ there exists $T >
0$ such that \eqref{Cauchy2} has a unique solution $\bb U(t)=(\bb u(t),\bb v(t))\in C([0, T], X_k)$.
\end{lemma}
\begin{proof}
It is sufficient to prove that  the corresponding $C_0$ semigroup $e^{tA}$ preserves the space $X_k$, that is, $e^{tA}X_k\subseteq X_k$.
Equivalently this fact means that the solution to the Cauchy problem 
\begin{align}\label{Cauchy_linear}
\left\{\begin{array}{c}
\partial_t\bb U(t)=A\bb U(t),\\
\bb U(0)=\bb U_0,\qquad\qquad
\end{array}\right.
\end{align}
belongs to $C([0, T], X_k)$ for $\bb U_0\in X_k$. Suppose that $k\geq 2$. Let $$\bb U(t)=e^{tA}\bb U_0=\Big((u_j(t))_{j=1}^N,(v_j(t))_{j=1}^N\Big)=\Big((u_1(t),\ldots, u_N(t)),(v_1(t),\ldots, v_N(t))\Big)$$  be a solution to \eqref{Cauchy_linear}. Then the function 
$$\bb V(t)=\Big((u_2(t),u_1(t), u_3(t),\ldots, u_N(t)),(v_2(t),v_1(t), v_3(t)\ldots, v_N(t))\Big)$$ is a solution to \eqref{Cauchy_linear} as well. Indeed, the linear equation in \eqref{Cauchy_linear} is invariant under the transposition of two first  elements of the vector solution $\bb U(t)$. By uniqueness $\bb U(t)=\bb V(t),$ therefore, $u_1(t)=u_2(t)$ and $v_1(t)=v_2(t).$ Repeating the process, one gets $u_1(t)=\ldots=u_k(t),\,\, v_1(t)=\ldots=v_k(t).$
  
\end{proof}
\begin{remark} 
The invariance property  $e^{tA}X_k\subseteq X_k$ might be alternatively shown by involving functional calculus. 

By \cite[Chapter I, Corollary 7.5]{Paz83}, for $\bb W\in \dom(A^2)$ we have
\begin{equation}\label{pazy}
e^{tA}\bb W=\frac{-1}{2\pi i}\int\limits_{\gamma-i\infty}^{\gamma+i\infty}e^{\lambda t} (A-\lambda)^{-1}\bb Wd\lambda,
\end{equation}
where $\gamma>\beta$ with $\beta$  from \eqref{est_group}.
Let $\lambda\in (\beta,\infty)$. We have
$$A-\lambda=\left(\begin{array}{cc}
-\lambda Id_{\mc E(\Gamma)} & Id_{L^2(\Gamma)}\\
-H_\alpha-m^2 & -\lambda Id_{L^2(\Gamma)}
\end{array}\right).$$
By a direct computation with operator-valued matrices,
\begin{equation}\label{resolvent}
(A-\lambda)^{-1}=\left(\begin{array}{cc}
-\lambda (H_\alpha+\lambda^2+m^2)^{-1} & -(H_\alpha+\lambda^2+m^2)^{-1}\\
(H_\alpha+m^2)(H_\alpha+\lambda^2+m^2)^{-1}  & -\lambda (H_\alpha+\lambda^2+m^2)^{-1}
\end{array}\right).\end{equation}
Observe that for $\beta$ large enough  $-\lambda^2\in\rho(H_\alpha+m^2).$
Using formula $(17)$ in \cite{BanIgn14} and denoting $z=\sqrt{m^2+\lambda^2}$, we get
\begin{equation*}
((H_\alpha+\lambda^2+m^2)^{-1}\bb w)_j=\tilde{c}_je^{-zx}+\frac{1}{2z}\int\limits_0^\infty w_j(y)e^{-|x-y|z}dy. 
\end{equation*}
Analogously to \cite[Lemma 2.3]{AngGol18a},
if $\bb w\in L^2_k(\Gamma),$ then $\tilde{c}_1=\ldots=\tilde{c}_k$ and $\tilde{c}_{k+1}=\ldots=\tilde{c}_N,$ consequently $(H_\alpha+\lambda^2+m^2)^{-1}\bb w\in L^2_k(\Gamma)$. Hence, by \eqref{resolvent},  $(A-\lambda)^{-1}\bb W\in \dom(A)\cap X_k$ for $\bb W\in X_k$. From \eqref{pazy} we get that $e^{tA}(\dom(A^2)\cap X_k)\subseteq X_k.$ 
By \cite[Chapter I, Theorem 2.7]{Paz83}, $\overline{\dom(A^2)}=X$ which implies   $\overline{\dom(A^2)\cap X_k}=X_k$, therefore, $e^{tA}X_k\subseteq X_k$.\hfill$\square$
 
\end{remark}

\section{Stability properties of standing waves}
In this section we  study stability/instability of the  standing waves $e^{i\omega t}\bs{\varphi}_{k,\omega}^\alpha$, where
 $\bs{\varphi}_{k,\omega}^\alpha$ is defined by \eqref{Phi_k}.
Orbital stability is understood in the following sense. 

\begin{definition}\label{def_stabil}
 The standing wave $\mathbf u(t, x) = e^{i\omega t}\bs{\varphi}(x)$  is said to be \textit{orbitally stable}  if for any $\varepsilon > 0$ there exists $\eta > 0$ with the following property. If $(\mathbf u_0,\bb u_1) \in X$  satisfies  $\|(\mathbf u_0,\bb u_1)-(\bs{\varphi}, i\omega \bs{\varphi})\|_{X} <\eta$,
then the solution $\mathbf U(t)$ of \eqref{Hamil} with $\bb U(0) = (\mathbf u_0,\bb u_1)$ exists globally, and
\[\sup\limits_{t\in [0,\infty)}\inf\limits_{\theta\in\mathbb{R}}\|\bb U(t)-e^{i\theta}(\bs{\varphi}, i\omega \bs{\varphi})\|_{X} < \varepsilon.\]
Otherwise, the standing wave  $\mathbf u(t, x) = e^{i\omega t}\bs{\varphi}(x)$ is said to be \textit{orbitally unstable}.
\end{definition}
In the sequel we will use the notation $\bb \Phi_{k,\omega}^\alpha=(\bs{\varphi}_{k,\omega}^\alpha, i\omega\bs{\varphi}_{k,\omega}^\alpha).$ 
\subsection{Stability approach}
Below we will introduce basic ingredients   of the classical theory by \cite{GrilSha87, GrilSha90} (see also \cite{BieGen15, Stu08}).
The key object is the Lyapunov functional  $S_\omega\in C^2(X,\mathbb{R})$ defined by
$$S_\omega(\bb u,\bb v)=E(\bb u,\bb v)+\omega Q(\bb u,\bb v).$$
From \eqref{func_deriv} one concludes that $\bb \Phi_{k,\omega}^\alpha$ is a critical point of $S_\omega$.
Let $\mc R: X\to X^*$ be the Riesz isomorphism. 
A principal role in the stability/instability study is  played  by the spectral properties of the  operator $\mc R^{-1}S''_\omega(\bb \Phi_{k,\omega}^\alpha): X\to X$. In what follows we will denote $\mc L^\alpha_k:=\mc R^{-1}S''_\omega(\bb \Phi_{k,\omega}^\alpha)$.
  Since $S''_\omega(\bb \Phi_{k,\omega}^\alpha): X\to X^*$ is bounded and symmetric, the operator $ \mc L^\alpha_k: X\to X$ is bounded and self-adjoint, i.e.
$$\la  \mc L^\alpha_k\bb U,\bb V\ra_X=\la S''_\omega(\bb \Phi_{k,\omega}^\alpha)\bb U,\bb V\ra_{X^*\times X}=\la S''_\omega(\bb \Phi_{k,\omega}^\alpha)\bb V, \bb U \ra_{X^*\times X}=\la \bb U, \mc L^\alpha_k\bb V\ra_X, \quad \bb U,\bb V\in X.$$
 Above we also have used  the fact that $X$ is a real Hilbert space. 
  
We consider the following list of assumptions about the spectrum of $\mc L^\alpha_k:$
\begin{itemize}
\item[$(A1)$] $n(\mc L^\alpha_k)=1$;
\item[$(A2)$] $n(\mc L^\alpha_k)=2$;
\item[$(A3)$] $\ker(\mc L^\alpha_k)=\Span\{i\bb \Phi_{k,\omega}^\alpha\}$;
\item[$(A4)$] apart from the non-positive eigenvalues, $\sigma(\mc L^\alpha_k)$ is positive and bounded away from zero.
\end{itemize}
We also define the notion of linear instability.
\begin{definition} The standing wave $e^{i\omega t} \bs{\varphi}_{k,\omega}^\alpha$ is \textit{linearly unstable} if  $\bb 0$ is a linearly unstable solution  for  the linearized  equation
\begin{equation*}\label{lineariz}
\mc J \partial_t \bb V(t)=S''_\omega(\bb \Phi_{k,\omega}^\alpha)\bb V(t)
\end{equation*} 
in the sense of Lyapunov.

\end{definition}
Due to \cite{GrilSha87, GrilSha90}, one can formulate the following stability/istability result.
\begin{theorem}\label{thm_stabil}
Let  assumptions $(A3), (A4)$  be valid, then the following two assertions hold.
\begin{itemize}
\item[$(i)$]  Suppose that  $\partial_\omega Q(\bb \Phi_{k,\omega}^\alpha)|_{\omega=\omega_0}>0$. 

$\bullet$ If, in addition, the assumption $(A1)$ holds, then the standing wave $e^{i\omega_0t} \bs{\varphi}_{k,\omega_0}^\alpha$ is orbitally stable. 

$\bullet$  If, in addition, the assumption $(A2)$ holds, then the standing wave $e^{i\omega_0 t}\bs{\varphi}_{k,\omega_0}^\alpha$ is linearly unstable.  
 
\item[$(ii)$] Suppose that  $\partial_\omega Q(\bb \Phi_{k,\omega}^\alpha)|_{\omega=\omega_0}<0$ and $(A1)$ holds, then the standing wave $e^{i\omega_0 t}\bs{\varphi}_{k,\omega_0}^\alpha$ is orbitally unstable. 
\end{itemize}
\end{theorem}

It is standard to verify that for $(\bb u,\bb v)\in X$
\begin{equation*}\label{second_deriv}S''_\omega(\bb \Phi_{k,\omega}^\alpha)(\bb u,\bb v)=\left(\widetilde H_\alpha\bb u+m^2\bb u-(\bs{\varphi}_{k,\omega}^\alpha)^{p-1}\bb u-(p-1)(\bs{\varphi}_{k,\omega}^\alpha)^{p-1}\re(\bb u)+i\omega\bb v, \bb v-i\omega \bb  u\right).\end{equation*}
Here the operator $\widetilde H_\alpha$ is understood in the following sense: since bilinear form $t_\alpha(\bb u_1, \bb u_2)=\la\bb u'_1, \bb u'_2\ra_{L^2(\Gamma)}+\alpha\re(u_{11}(0) \overline{u_{12}}(0))$ is  bounded  on $\EE(\Gamma)$, there exists a unique bounded operator $\widetilde{H}_\alpha: \EE(\Gamma)\to\EE^*(\Gamma)$ such that $t_\alpha(\bb u_1, \bb u_2)=\la\widetilde{H}_\alpha\bb u_1, \bb u_2\ra_{\EE(\Gamma)^*\times \EE(\Gamma)}.$

Commonly   in stability study one substitutes the operator $ \mc L^\alpha_k$ acting on $X$ by the  self-adjoint operator acting in $L^2(\Gamma)\times L^2(\Gamma)$ (with  the real inner product). Namely, this operator is associated (by the Representation Theorem \cite[Chapter VI, Theorem 2.1]{Kat66}) with the closed, densely defined, bounded from below bilinear form  
\begin{equation}\label{bil_form}
\begin{split}&b_k(\bb U,\bb V)=\la S''_\omega(\bb \Phi_{k,\omega}^\alpha)\bb U,\bb V\ra_{X^*\times X}\\
&=\re\left[\alpha u_1(0)\overline{w_1(0)}+\int\limits_\Gamma\Big(\bb u'\overline{\bb w}'+m^2\bb u\overline{\bb w}-(\bs{\varphi}_{k,\omega}^\alpha)^{p-1}\bb u\overline{\bb w}-(p-1)(\bs{\varphi}_{k,\omega}^\alpha)^{p-1}\re({\bb u})\overline{\bb w}\Big)dx\right]\\&+\re\left[\int\limits_\Gamma\bb v\overline{\bb z}dx+
\omega\int\limits_\Gamma (i\bb v\overline{\bb w}-i\bb u\overline{\bb z})dx\right], \quad \bb U=(\bb u,\bb v), \bb V=(\bb w,\bb z)\in X.
\end{split}\end{equation}
\begin{proposition}\label{spec}
The self-adjoint operator associated in $L^2(\Gamma)\times L^2(\Gamma)$ with the bilinear form $b_k(\bb U,\bb V)$ is given by 
\begin{equation*}
\begin{split}
&L^\alpha_k(\bb u,\bb v)=\left(H_\alpha\bb u+m^2\bb u-(\bs{\varphi}_{k,\omega}^\alpha)^{p-1}\bb u-(p-1)(\bs{\varphi}_{k,\omega}^\alpha)^{p-1}\re(\bb u)+i\omega\bb v, \bb v-i\omega \bb  u\right),\\
&\dom(L^\alpha_k)=\dom(H_\alpha)\times L^2(\Gamma).
\end{split}
\end{equation*}
Moreover, $\ker(\mc L^\alpha_k)=\ker(L^\alpha_k)$, \, $n(\mc L^\alpha_k)=n(L^\alpha_k)$,  and  $$\inf \sigma_{\mathrm{ess}}(L^\alpha_k)>0 \,\,\, \Rightarrow\,\,\, \inf \sigma_{\mathrm{ess}}(\mc L^\alpha_k)>0.$$
\end{proposition}
\begin{proof}
Denote by $\mc T_k$ the self-adjoint operator associated with the bilinear form $b_k(\bb U,\bb V)$. Then 
\begin{align*}
&\dom(\mc T_k)=\left\{\bb U\in X: \exists \bb W\in L^2(\Gamma)\times L^2(\Gamma)\,\, s. t.\,\, \forall\bb V\in X,\, \,b_k(\bb U,\bb V)=\la\bb W,\bb V\ra_{L^2(\Gamma)\times L^2(\Gamma)} \right\},\\
&\mc T_k\bb U=\bb W.
\end{align*}
It is easily seen that $\dom(L^\alpha_k)\subseteq \dom(\mc T_k)$ and $L^\alpha_k\bb U=\mc T_k\bb U$ for $\bb U\in \dom(L^\alpha_k).$ Indeed, $b_k(\bb U,\bb V)=\la L^\alpha_k\bb U,\bb V\ra_{L^2(\Gamma)\times L^2(\Gamma)}$  for all $\bb U\in \dom(L^\alpha_k)$ and $\bb V\in X$. We need to show that  $\dom(\mc T_k)\subseteq \dom(L^\alpha_k)$. Let $\bb {\hat U}=(\bb{\hat u},\bb{\hat v})=\left((\hat u_j)_{j=1}^N, (\hat v_j)_{j=1}^N\right)\in \dom(\mc T_k)$. Observe that, by the definition of $\dom(\mc T_k)$,  the functional $f_k(\bb V)=b_k(\bb {\hat U},\bb V)=\la\bb{\hat W},\bb V\ra_{L^2(\Gamma)\times L^2(\Gamma)}$ is linear and continuous on $L^2(\Gamma)\times L^2(\Gamma)$. 
Given $\bb V=(\bb w, \bb 0)$ with $\bb w=(w_j)_{j=1}^N\in C_0^\infty(\Gamma)=\bigoplus\limits_{j=1}^N C_0^\infty(\mathbb{R_+})$, then integrating by parts in \eqref{bil_form} and using  continuity of $f_k$, one gets that $\bb {\hat u}\in H^2(\Gamma).$ 
Finally, observing  $\bb{\hat u}\in H^2(\Gamma)$, integrating by parts in \eqref{bil_form} with $\bb w\in \mc E(\Gamma)$ such that $w_1(0)\neq 0$, and using continuity of $f_k$  again, we arrive at the conclusion that
$$g_k(\bb V)=\re\left[\Big(\alpha \hat u_1(0)-\sum\limits_{j=1}^N\hat u'_j(0)\Big)\overline{w_1(0)}\right]$$
has to be a continuous functional on $L^2(\Gamma)\times L^2(\Gamma)$. This is true only if $\alpha \hat u_1(0)-\sum\limits_{j=1}^N\hat u'_j(0)=0$, therefore, $\bb {\hat u}\in\dom(H_\alpha)$ and $\bb{\hat U}\in\dom(L^\alpha_k).$

The second part of the proposition follows by \cite[Lemma 5.4]{Stu08}. To apply  Lemma 5.4, we  only need to prove that  inequality $(G)$  (G\r{a}rding's-type inequality) holds, that is,  there exist $\varepsilon, C>0$ such that 
\begin{equation}\label{Garding} 
\la S''_\omega(\bb \Phi_{k,\omega}^\alpha)\bb V,\bb V\ra_{X^*\times X}\geq\varepsilon \|\bb V\|^2_X-C\|\bb V\|^2_{L^2(\Gamma)\times L^2(\Gamma)},\quad \bb V=(\bb w, \bb z)\in X.
\end{equation}
From \eqref{bil_form} we get
\begin{equation}\label{est_1}
\begin{split}
\la S''_\omega(\bb \Phi_{k,\omega}^\alpha)\bb V,\bb V\ra_{X^*\times X}
&=\alpha|w_1(0)|^2+\int\limits_\Gamma\Big(|\bb w'|^2+m^2|\bb w|^2-(\bs{\varphi}_{k,\omega}^\alpha)^{p-1}|\bb w|^2-(p-1)(\bs{\varphi}_{k,\omega}^\alpha)^{p-1}(\re{\bb w})^2\Big)dx+ \int\limits_\Gamma|\bb z|^2dx\\&-2
\omega\int\limits_\Gamma\im (\bb z\overline{\bb w})dx.
\end{split}\end{equation}
Moreover, by  \eqref{sigma_inf},  we deduce
 $$\alpha|w_1(0)|^2+\frac{1}{2}\int\limits_\Gamma|\bb w'|^2dx=\frac{1}{2}\Big(2\alpha|w_1(0)|^2+\int\limits_\Gamma|\bb w'|^2dx\Big)\geq \frac{-2\alpha^2}{N^2}\|\bb w\|_{L^2(\Gamma)}.$$
Denoting $M=\|\bs{\varphi}_{k,\omega}^\alpha\|_{L^\infty(\Gamma)}$, we obtain from \eqref{est_1}
\begin{align*}
&\la S''_\omega(\bb \Phi_{k,\omega}^\alpha)\bb V,\bb V\ra_{X^*\times X}\geq\int\limits_\Gamma\Big(\frac{1}{2}|\bb w'|^2+m^2|\bb w|^2+|\bb z|^2\Big)dx-\int\limits_\Gamma\left(pM^{p-1}|\bb w|^2-2\omega|\bb z\|\bb w|\right)dx-\frac{2\alpha^2}{N^2}\|\bb w\|_{L^2(\Gamma)} \\&\geq \min\{\frac{1}{2}, m^2\}\left(\|\bb w\|^2_{H^1(\Gamma)}+\|\bb z\|^2_{L^2(\Gamma)}\right)-\Big(pM^{p-1}+\frac{2\alpha^2}{N^2}+|\omega|\Big)\left(\|\bb w\|^2_{L^2(\Gamma)}+\|\bb z\|^2_{L^2(\Gamma)}\right),
\end{align*}
and therefore \eqref{Garding} holds.
\end{proof}
\subsection{Spectral properties of $L^\alpha_k$}
Another standard step in the stability study  is to express the operator $L^\alpha_k$ in more convenient form using two operators acting on real-valued functions. Let $\bb U=(\bb u,\bb v)\in \dom(L^\alpha_k)$  and $\bb u=\bb u_1+i\bb u_2, \,\bb v=\bb v_1+i\bb v_2$ with real-valued vector functions $\bb u_j, \bb v_j,\, j\in\{1,2\}.$
We have 
\begin{align*}
&L^\alpha_k(\bb u,\bb v)=\left(\begin{array}{cc} -\dfrac{d^2}{dx^2}+m^2-(\bs{\varphi}_{k,\omega}^\alpha)^{p-1}-(p-1)(\bs{\varphi}_{k,\omega}^\alpha)^{p-1}\re(\cdot)& i\omega \\
-i\omega  & 1
\end{array}\right)\left(\begin{array}{c}\bb u_1+i\bb u_2\\
\bb v_1+i\bb v_2
\end{array}\right)\\&=\left(\begin{array}{c}\left(-\bb u''_1+m^2\bb u_1-p(\bs{\varphi}_{k,\omega}^\alpha)^{p-1}\bb u_1-\omega\bb v_2\right)+i\left(-\bb u''_2+m^2\bb u_2-(\bs{\varphi}_{k,\omega}^\alpha)^{p-1}\bb u_2+\omega\bb v_1\right)\\
\left(\omega \bb u_2+\bb v_1\right)+i\left(-\omega\bb u_1+\bb v_2\right)
\end{array}\right).
\end{align*}
Substituting the  complex-valued vector function $\bb U=(\bb u,\bb v)$ by the corresponding quadruplet of real-valued functions $(\bb u_1, \bb u_2, \bb v_1, \bb v_2),$  and substituting  $L^\alpha_k(\bb u,\bb v)=(\bb f,\bb g)=(\bb f_1+i\bb f_2, \bb g_1+i\bb g_2)\in L^2(\Gamma)\times L^2(\Gamma)$ by the quadruplet $(\bb f_1, \bb f_2, \bb g_1, \bb g_2)$, we can interpret the operator $L^\alpha_k$  as 
\begin{align}\label{repres}
\left(\begin{array}{c}
\bb f_1\\
\bb f_2\\
\bb g_1\\
\bb g_2
\end{array}\right)=\left(\begin{array}{cccc}
L_{1,k}^\alpha +\omega^2& 0 & 0 & -\omega\\
0 & L_{2,k}^\alpha +\omega^2 & \omega & 0\\
0 & \omega & 1 & 0\\
-\omega & 0 & 0 & 1 
\end{array}\right)\left(\begin{array}{c}
\bb u_1\\
\bb u_2\\
\bb v_1\\
\bb v_2
\end{array}\right),
\end{align} 
where 
\begin{equation*}
\begin{split}
&L_{1,k}^\alpha\bb u=-\bb u''+(m^2-\omega^2)\bb u-p(\bs{\varphi}_{k,\omega}^\alpha)^{p-1}\bb u,\\
&L_{2,k}^\alpha\bb u=-\bb u''+(m^2-\omega^2)\bb u-(\bs{\varphi}_{k,\omega}^\alpha)^{p-1}\bb u,\quad \dom(L_{j,k}^\alpha)=\dom(H_\alpha),\, j\in\{1,2\}.
\end{split}
\end{equation*}
\begin{remark}\label{ker_L} From \eqref{repres} we deduce $$(\bb u,\bb v)=(\bb u_1+i\bb u_2, \bb v_1+i\bb v_2)\in \ker(L^\alpha_k)\quad \Longleftrightarrow\quad\left\{\begin{array}{c}
\bb u_1\in \ker(L_{1,k}^\alpha),\\
\bb u_2\in \ker(L_{2,k}^\alpha),\\
\bb v_1=-\omega \bb u_2,\,\,\quad\,\\
\bb v_2=\omega\bb u_1.\qquad\,
\end{array}\right.$$
\end{remark}

Assuming that  the operators $L_{1,k}^\alpha$ and $L_{2,k}^\alpha$ act as usual on  $L^2(\Gamma, \mathbb{C})$ with the usual complex structure, one can prove the following result on the relation between their spectra and  the spectrum of $L^\alpha_k$.
\begin{proposition}\label{spect_L_j}
Let $\lambda\in\mathbb{R}\setminus\{1\}$ and $\mu(\lambda):=\lambda+\frac{\lambda\omega^2}{1-\lambda}$.  Assume also that $k\in\left\{0,\ldots,\left[\tfrac{N-1}{2}\right]\right\}$ and  $m^2-\omega^2>\tfrac{\alpha^2}{(N-2k)^2}$. Then the following assertions hold.
\begin{itemize}
\item[$(i)$] $\lambda\in \sigma(L^\alpha_k)\quad \Longleftrightarrow\quad\mu(\lambda)\in \sigma(L_{1,k}^\alpha)\cup \sigma(L_{2,k}^\alpha)$.

\item[$(ii)$] $\dim(\ker(L^\alpha_k-\lambda))=\dim(\ker(L_{1,k}^\alpha-\mu(\lambda)))+\dim(\ker(L_{2,k}^\alpha-\mu(\lambda)))$, 
consequently $n(L^\alpha_k)=n(L_{1,k}^\alpha)+n(L_{2,k}^\alpha)$.
\item[$(iii)$] $\lambda\in \sigma_{\ess}(L^\alpha_k)\quad\Longleftrightarrow\quad\mu(\lambda)\in \sigma_{\ess}(L_{1,k}^\alpha)\cup \sigma_{\ess}(L_{2,k}^\alpha)$.

\item[$(iv)$] Let $\mu_-(\lambda)$ and  $\mu_+(\lambda)$ be the restrictions of $\mu(\lambda)$ to $(-\infty, 1)$ and $(1, \infty)$ respectively. Then 
\begin{equation}\label{inverse} \sigma_{\ess}(L^\alpha_k)\setminus\{1\}=\mu_-^{-1}\left(\sigma_{\ess}(L_{1,k}^\alpha)\cup \sigma_{\ess}(L_{2,k}^\alpha)\right)\cup \mu_+^{-1}\left(\sigma_{\ess}(L_{1,k}^\alpha)\cup \sigma_{\ess}(L_{2,k}^\alpha)\right).
\end{equation}

\item[$(v)$]  $\sigma_{\ess}(L_{1,k}^\alpha)\cup \sigma_{\ess}(L_{2,k}^\alpha)=\sigma_{\ess}(L_{1,k}^\alpha)=\sigma_{\ess}(L_{2,k}^\alpha)=[m^2-\omega^2,\infty)$.

\item[$(vi)$] $ \sigma_{\ess}(L^\alpha_k)=[\sigma_1,1]\cup[\sigma_2,\infty)$,
where $\sigma_1=\mu_-^{-1}\left(m^2-\omega^2\right)\in (0,1)$ and $\sigma_2=\mu_+^{-1}\left(m^2-\omega^2\right)\in (1,\infty).$

\end{itemize}
\end{proposition}
\begin{proof}
The proof of items $(i)-(iii)$ repeats the proof  of \cite[Proposition 4.5]{CsoGen18} (one just needs to substitute the operator $\mathcal{L}_\beta$ by $L_k^\alpha$,  the operator $L_\beta^-$  by $L_{2,k}^\alpha$, and the  operator $L_\beta^+$ by $L_{1,k}^\alpha$). The key point is that for $\lambda\neq 1$,
\begin{align*}
(L^\alpha_k-\lambda)(\bb u,\bb v)=(\bb f, \bb g)\quad &\Longleftrightarrow \quad \left\{\begin{array}{c}
(L_{1,k}^\alpha+\omega^2-\lambda)\bb u_1-\omega \bb v_2 = \bb f_1,\\ 
(L_{2,k}^\alpha+\omega^2-\lambda)\bb u_2+\omega \bb v_1 = \bb f_2,\\
\qquad\qquad(1-\lambda)\bb v_1+\omega \bb u_2 = \bb g_1,\\
\qquad\qquad(1-\lambda)\bb v_2-\omega \bb u_2 = \bb g_2.
\end{array} \right.\\
&\Longleftrightarrow \quad \left\{\begin{array}{c}
(L_{1,k}^\alpha-\mu(\lambda))\bb u_1 = \bb f_1+\frac{\omega}{1-\lambda}\bb g_2,\\ 
(L_{2,k}^\alpha-\mu(\lambda))\bb u_2 = \bb f_2-\frac{\omega}{1-\lambda}\bb g_1,\\
\quad\qquad\qquad\bb v_1=\frac{1}{1-\lambda}(\bb g_1-\omega\bb u_2),\\
\quad\qquad\qquad\bb v_2=\frac{1}{1-\lambda}(\bb g_2+\omega\bb u_1).
\end{array} \right.
\end{align*}
Item $(iv)$ follows from $(iii)$ and the fact that $\mu_-$ and $\mu_+$ are increasing bijections. 

 Item $(v)$ seems  natural, but we didn't manage to find its proof in the literature.  Firstly, consider the self-adjoint operator 
  \begin{equation*}
\begin{split}
(H_{\infty}\mathbf{v})(x)&=\left(-v_j''(x)\right)_{j=1}^N,\quad x> 0,\quad  \mathbf{v}=(v_j)_{j=1}^N,\\
\dom(H_\infty)&=\Big\{\mathbf{v}\in H^2(\Gamma): v_1(0)=\ldots=v_N(0)=0\Big\}.
\end{split}
\end{equation*}
Observe that $H_\infty=\bigoplus\limits_{j=1}^Nh_\infty,$
where \begin{equation*}
\begin{split}
(h_{\infty}{v})(x)=-v''(x),\quad x> 0,\quad
\dom(h_\infty)=\Big\{{v}\in H^2(\mathbb{R}_+): v(0)=0\Big\}.
\end{split}
\end{equation*}
 Therefore, $\sigma_{\ess}(H_\infty)=\sigma_{\ess}(h_\infty)=[0,\infty).$
Secondly, notice that  the operator $H_\alpha$ defined  by \eqref{D_alpha} and the operator $H_\infty$ are self-adjoint extensions of the symmetric operator 
 \begin{equation*}
\begin{split}
(\tilde H_0\mathbf{v})(x)&=\left(-v_j''(x)\right)_{j=1}^N,\quad x> 0,\quad  \mathbf{v}=(v_j)_{j=1}^N,\\
\dom(\tilde H_0)&=\Big\{\mathbf{v}\in H^2(\Gamma): v_1(0)=\ldots=v_N(0)=0,\,\,\sum\limits_{j=1}^N  v_j'(0)=0\Big\}.
\end{split}
\end{equation*}
The operator has equal deficiency indices $n_{\pm}(\tilde H_0)=1$  (see \cite[proof of Theorem 3.5-$(iii)$]{AngGol18a}), therefore, by Krein's resolvent formula, the operator $(H_\alpha-\lambda)^{-1}-(H_\infty-\lambda)^{-1},\,\,\lambda\in\rho(H_\infty)\cap\rho(H_\alpha),$ is of rank one (see \cite[Appendix A, Theorem A.2]{AlbGes05}). Then, by Weyl's theorem \cite[Theorem XIII.14]{ReeSim78},  $\sigma_{\ess}(H_\alpha)=\sigma_{\ess}(H_\infty)=[0,\infty)$, and consequently $\sigma_{\ess}(H_\alpha+m^2-\omega^2)=[m^2-\omega^2,\infty)$. The operator of multiplication by $(\bs{\varphi}_{k,\omega}^\alpha)^{p-1}$ is relatively $(H_\alpha+m^2-\omega^2)$-compact (for the idea of the proof see, for instance, \cite[Proposition 8.20]{Sch12}). Therefore,  $\sigma_{\ess}(L^\alpha_{1,k})=\sigma_{\ess}(L^\alpha_{2,k})=[m^2-\omega^2,\infty)$ (see Corollary 2 of \cite[Theorem XIII.14]{ReeSim78}).

Finally, by $(v)$, $\sigma_{\ess}(L_{1,k}^\alpha)\cup \sigma_{\ess}(L_{2,k}^\alpha)$ contains a neighborhood of $+\infty$. Since  $\sigma_{\ess}(L^\alpha_k)$ is closed,  \eqref{inverse} yields  $(vi).$ 

\end{proof}
\begin{remark}
The equality $\sigma_{\ess}(H_\alpha)=[0,\infty)$ might be shown using classical Weyl's criterion (see \cite[Proposition 8.11]{Sch12}). 
Let  $H_0$ be defined by \eqref{Kirchhoff}. It is easily seen that $H_0\geq 0$, therefore, $\sigma_{\ess}(H_0)\subseteq\sigma(H_0)\subseteq [0,\infty)$. We can prove that $[0,\infty)\subseteq\sigma_{\ess}(H_0).$

Given $\lambda > 0$. Then, by  Weyl's criterion,  $\lambda \in \sigma_{\ess}(H_0)$ if, and  only if, there exists a sequence $\{\bb x_n\}_{n=1}^\infty\subset\dom(H_0)$ such that 

\begin{equation}\label{W-seq} \lim\limits_{n \to \infty}\frac{\|(H_0-\lambda)\bb x_n\|_{L^2(\Gamma)}}{\|\bb x_n\|_{L^2(\Gamma)}}=0,\end{equation}
 and $\bb x_n/\|\bb x_n\|_{L^2(\Gamma)}$  tends weakly to $\bb 0$ in $L^2(\Gamma)$.

We fix a function $\phi(x)\in C^\infty_0(\mathbb{R}_+)$ such that 
$$\phi(x)\geq 0,\quad  \phi(x)=1 \,\,\text{for}\,\, 1/4\leq x\leq 1/2,\quad\text{and}\quad\phi(x) = 0\,\, \text{for}\,\, x\geq 1.$$ 
We set $$\phi_n(x) = \phi(\frac 1{\sqrt{n}}(x-n^2)),\quad n\in \mathbb{N}.$$
Then 
\begin{equation*}\label{supp}
\supp \phi_n\subseteq (n^2,n^2+\sqrt{n}),\,\, \text{and}\,\,\supp \phi_k\cap\supp \phi_j=\emptyset,\,\, \text{for}\,\,\,k \neq j,\,\,\, k,j\in \mathbb{N}.\end{equation*} 
It is easily seen that 
$$\bb x_n=\left(e^{i\sqrt{\lambda}x}\phi_n(x),\underset{\bb 2}{0},\ldots,\underset{\bb N}{0}\right)$$ serves for \eqref{W-seq}.
Hence $\lambda\in \sigma_{\ess}(H_0)$ and  $\sigma_{\ess}(H_0)=[0,\infty)$ by closedness of the essential spectrum.  Then, by Weyl's theorem \cite[Theorem XIII.14]{ReeSim78},  $\sigma_{\ess}(H_\alpha)=\sigma_{\ess}(H_0)=[0,\infty)$.\hfill$\square $
\end{remark}
\subsection{Spectral properties of $L_{1,k}^\alpha$  and $L_{2,k}^\alpha$}
\begin{proposition}\label{grafoN2}
Let $\alpha\neq 0$,   $k\in\left\{0,\ldots,\left[\tfrac{N-1}{2}\right]\right\}$, and  $m^2-\omega^2>\tfrac{\alpha^2}{(N-2k)^2}$. Then
\begin{itemize}
\item[$(i)$]   $\ker(L^\alpha_{2,k})=\Span\{\bs{\varphi}_{k,\omega}^\alpha\}$ and $L^\alpha_{2,k}\geq 0$;
 \item[$(ii)$] $\ker(L^\alpha_{1,k})=\{\mathbf{0}\}$;
 \item[$(iii)$] $\ker(L^\alpha_k)=\{i\bb \Phi^\alpha_{k,\omega}\}.$
 
  \end{itemize}
\end{proposition}
\begin{proof}
For the proof of $(i), (ii)$ see \cite[Proposition 1]{AngGol18a} (with $\omega$ substituted by $m^2-\omega^2$). The proof of
 $(iii)$ follows from $(i),(ii)$, and Remark \ref{ker_L}.      
         
\end{proof}

The description of the negative spectrum of $L^\alpha_{1,k}$  might be obtained as in \cite[Theorem 3.4]{AngGol18a}  and \cite[Proposition 3.17]{AngGol18}. For the  reader's convenience  we provide the principal steps of the proofs.
 
 Consider  the following  self-adjoint  Schr\"odinger operator on $L^2(\Gamma)$ with  the Kirchhoff condition at $\nu=0$,
 \begin{equation}\label{L^0_1}
 \begin{split}
 &L^0_1\bb v=-\bb v''+(m^2-\omega^2)\bb v-p\varphi_0^{p-1}\bb v, \\
 &\dom(L^0_1)=\Big\{\mathbf{v}\in H^2(\Gamma): v_1(0)=\ldots=v_N(0),\,\,\sum\limits_{j=1}^N  v_j'(0)=0\Big\},
 \end{split}
 \end{equation}
 where $\varphi_0$ is the half-soliton solution for the classical NLS model,
 \begin{equation*}\label{var_0}
 \varphi_0(x)=\left[\frac{(p+1)(m^2-\omega^2)}{2} \sech^2\left(\frac{(p-1)\sqrt{m^2-\omega^2}}{2}x \right)\right]^{\frac{1}{p-1}}.
  \end{equation*}
From the definition of the profiles $\bs{\varphi}_{k,\omega}^\alpha$ in \eqref{Phi_k} one gets
 \begin{equation*}\label{converg}
 \bs{\varphi}_{k,\omega}^\alpha\underset{\alpha\to 0}{\longrightarrow} \bs{\varphi}_0\quad \text{in}\;\; H^1(\Gamma),
 \end{equation*}
 where $\bs{\varphi}_0=(\varphi_0)_{j=1}^N$.
To study negative spectrum of ${L}^\alpha_{1,k}$, we apply the analytic perturbation theory. Hence first we need to describe spectral properties of  $L^0_1$ (which is the limiting value of ${L}^\alpha_{1,k}$ as $\alpha\to 0$).

  \begin{theorem}\label{spect_L^0_1}
 Let $L^0_1$ be defined by \eqref{L^0_1} and $k\in\left\{1,\ldots,\left[\tfrac{N-1}{2}\right]\right\}$. Then 
\begin{itemize}
  \item[$(i)$] $\ker(L^0_1)=\Span\{\hat{\bs{\varphi}}_{0,1},\ldots,\hat{\bs{\varphi}}_{0,N-1}\}$, where
   \begin{equation*}\label{Psi_j}
\hat{\bs{\varphi}}_{0,j}=(0,\ldots,0,\underset{\bf j}{\varphi'_{0}},\underset{\bf j+1}{-\varphi'_{0}},0,\ldots,0);
 \end{equation*}
  \item[$(ii)$] in the space  $L^2_k(\Gamma)$ we have $\ker(L^0_1)=\Span\{\bs{\varphi}_{0,k}\}$, i.e. $\ker(L^0_1|_{L^2_k(\Gamma)})=\Span\{\bs{\varphi}_{0,k}\}$, where
  \begin{equation}\label{Psi_0}
\bs{\varphi}_{0,k}=\left(\underset{\bf 1}{\tfrac{N-k}{k}\varphi'_{0}},\ldots, \underset{\bf k}{\tfrac{N-k}{k}\varphi'_{0}},\underset{\bf k+1}{-\varphi'_{0}},\ldots,\underset{\bf N}{-\varphi'_{0}}\right);
 \end{equation}
  \item[$(iii)$] $n(L^0_1)=n(L^0_1|_{L^2_k(\Gamma)})=1$;
   \item[$(iv)$] the rest of the spectrum of $L^0_1$  is positive and bounded away from zero.
  \end{itemize}
  \end{theorem}
  \begin{proof}
  For the proof of $(i)-(iii)$ see \cite[Theorem 3.5]{AngGol18a}. As in the proof of Proposition \ref{spect_L_j}-$(v)$, one might show that $\sigma_{\ess}(L^0_1)=[m^2-\omega^2,\infty)$ and therefore $(iv)$ holds.
\end{proof}
One of the principal facts for  the investigation of the negative spectrum of the operator ${L}^\alpha_{1,k}$ is the following lemma.
 \begin{lemma}\label{analici}Let $k\in\left\{1,\ldots,\left[\frac{N-1}{2}\right]\right\}.$  As a function of $\alpha$, $({L}^\alpha_{1,k})$ is a real-analytic family of self-adjoint operators of type (B) in the sense of Kato.
\end{lemma}
The above lemma and Theorem \ref{spect_L^0_1} lead to the next  result.
\begin{proposition}\label{perteigen} Let  $k\in\left\{1,\ldots,\left[\tfrac{N-1}{2}\right]\right\}$. Then there  exist $\alpha_0>0$ and two analytic functions $\lambda_k : (-\alpha_0,\alpha_0)\to \mathbb R$ and $\bb f_k: (-\alpha_0,\alpha_0)\to L^2_k(\Gamma)$ such that
\begin{enumerate}
\item[$(i)$] $\lambda_k(0)=0$ and $\bb f_k(0)=\bs{\varphi}_{0,k}$, where $\bs{\varphi}_{0,k}$ is defined by \eqref{Psi_0};

\item[$(ii)$] for all $\alpha\in (-\alpha_0,\alpha_0)$, $\lambda_k(\alpha)$ is the  simple isolated second eigenvalue of $L^\alpha_{1,k}$ in $L^2_k(\Gamma)$, and $\bb f_k(\alpha)$ is the associated eigenvector for $\lambda_k(\alpha)$;

\item[$(iii)$] $\alpha_0$ can be chosen small enough to ensure that for  $\alpha\in (-\alpha_0,\alpha_0)$  the spectrum of $L_{1,k}^\alpha$ in $L^2_k(\Gamma)$ is positive, except at most the  first two eigenvalues.
\end{enumerate}
\end{proposition}
\begin{proof}
It is implied  by the Kato-Rellich theorem (see \cite[Theorem XII.8]{ReeSim78}). For the details see the proof of \cite[Proposition 2]{AngGol18a}.
\end{proof}
The next proposition provides characterization of $n(L^\alpha_{1, k}|_{L^2_k(\Gamma)})$ for small $\alpha.$
\begin{proposition}\label{signeigen} 
 Let  $k\in\left\{1,\ldots,\left[\tfrac{N-1}{2}\right]\right\}$.
There exists $0<\alpha_1<\alpha_0$ such that $\lambda_k(\alpha)<0$ for any $\alpha\in (-\alpha_1,0)$, and $\lambda_k(\alpha)>0$ for  any $\alpha\in (0, \alpha_1)$. Therefore, $n(L^\alpha_{1, k}|_{L^2_k(\Gamma)})=2$  for $\alpha<0$, and $n(L^\alpha_{1, k}|_{L^2_k(\Gamma)})=1$  for $\alpha>0$ if $\alpha$ is small enough.
\end{proposition}
\begin{proof} From Taylor's theorem we have the  expansions
\begin{equation}\label{decomp1}
\lambda_k(\alpha)=\lambda_{0,k} \alpha+ O(\alpha^2)\quad\text{and}\quad \bb f_k(\alpha)=\bs{\varphi}_{0,k}+ \alpha \bb f_{0,k}  +   \bb O(\alpha^2),
\end{equation}
where  $\lambda_{0,k}=\lambda'_k(0)\in \mathbb R$ and $\bb f_{0,k}=\partial_\alpha \bb f_k(\alpha)|_{\alpha=0}\in L^2_k(\Gamma)$.  The desired result will follow  if we show that $\lambda_{0,k}>0$.  We compute $\la L^\alpha_{1,k} \bb f_k(\alpha), \bs{\varphi}_{0,k}\ra_{L^2(\Gamma)}$ in two different ways.

Note that for  $\bs{\varphi}_{k,\omega}^\alpha$ defined by \eqref{Phi_k} we have
\begin{equation}\label{1a}
\begin{split}
&\bs{\varphi}_{k,\omega}^\alpha=\bs{\varphi}_{0}+\alpha\bb g_{0,k}+ \bb O(\alpha^2),\\ \bb g_{0,k}=\partial_\alpha\bs{\varphi}_{k,\omega}^\alpha|_{\alpha=0}&=\frac{2}{(p-1)(N-2k)(m^2-\omega^2)}\left(\underset{\bf 1}{\varphi'_{0}},\ldots, \underset{\bf k}{\varphi'_{0}},\underset{\bf k+1}{-\varphi'_{0}},\ldots,\underset{\bf N}{-\varphi'_{0}}\right).
\end{split}
\end{equation}
 From \eqref{decomp1} we obtain
\begin{equation}\label{1}
\la L^\alpha_{1,k} \bb f_k(\alpha), \bs{\varphi}_{0,k}\ra_{L^2(\Gamma)}=\lambda_{0,k} \alpha\|\bs{\varphi}_{0,k}\|^2_{L^2(\Gamma)}+ O(\alpha^2).
\end{equation}
By  $L^0_{1}\bs{\varphi}_{0,k}=\bb 0$ and \eqref{decomp1} we get
\begin{equation}\label{2} L^\alpha_{1,k}\bs{\varphi}_{0,k}=p\left((\bs{\varphi}_0)^{p-1}-(\bs{\varphi}_{k,\omega}^\alpha)^{p-1}\right)\bs{\varphi}_{0,k}=-\alpha p(p-1)(\bs{\varphi}_{0})^{p-2}\bb g_{0,k}\bs{\varphi}_{0,k}+ \bb O(\alpha^2).
\end{equation}
The operations in  the last equality are  componentwise.
Equations \eqref{2},  \eqref{1a}, and $\bs{\varphi}_{0,k}\in \dom(H_\alpha)$ lead to
\begin{equation}\label{3}
\begin{split}
&\la L^\alpha_{1,k} \bb f_k(\alpha), \bs{\varphi}_{0,k}\ra_{L^2(\Gamma)}
=\la\bb f_k(\alpha),  L^\alpha_{1,k}\bs{\varphi}_{0,k}\ra_{L^2(\Gamma)}=-\la\bs{\varphi}_{0,k}, \alpha p(p-1)(\bs{\varphi}_{0})^{p-2}\bb g_{0,k}\bs{\varphi}_{0,k}\ra_{L^2(\Gamma)}+O(\alpha^2)\\&=-\alpha p(p-1)\left(\tfrac{(N-k)^2}{k}-(N-k)\right)\frac{2}{(p-1)(N-2k)(m^2-\omega^2)}\int\limits_0^\infty(\varphi'_0)^3\varphi_0^{p-2}dx+O(\alpha^2)\\&=-2\alpha p\frac{N-k}{k(m^2-\omega^2)}\int\limits_0^\infty(\varphi'_0)^3\varphi_0^{p-2}dx+O(\alpha^2).
\end{split}
\end{equation}
Finally, combining \eqref{3} and \eqref{1}, we obtain
$$\lambda_{0,k}=\dfrac{-2p\frac{N-k}{k(m^2-\omega^2)}\int\limits_0^\infty(\varphi'_0)^3\varphi_0^{p-2}dx}{\|\bs{\varphi}_{0,k}\|^2_{L^2(\Gamma)}}+O(\alpha).$$
It follows that $\lambda_{0,k}$ is positive for sufficiently small $|\alpha|$ (due to the negativity of $\varphi'_0$ on $\mathbb{R}_+$), which in view of \eqref{decomp1} ends the proof.
\end{proof}

Summarizing  the above results, we obtain the following characterization of the negative spectrum of   $L^\alpha_{1,k}$. 
\begin{proposition}\label{n(L_k)}
\begin{itemize}
\item[$(i)$]
Let $k\in\left\{1,\ldots,\left[\tfrac{N-1}{2}\right]\right\}$. 
\begin{enumerate}
\item[$(a)$] If  $\alpha>0$, then  $n(L^\alpha_{1,k}|_{L^2_k(\Gamma)})=1$.
\item[$(b)$] If $\alpha<0$, then  $n(L^\alpha_{1,k}|_{L^2_k(\Gamma)})=2$.
\end{enumerate}
\item[$(ii)$] Let $k=0$.
\begin{enumerate}
\item[$(a)$] If  $\alpha<0$, then   $n(L^\alpha_{1,0})=1$ in $L^2(\Gamma)$.
\item[$(b)$] If $\alpha>0$, then  $n(L^\alpha_{1,0}|_{L^2_{\eq}(\Gamma)})=1$.
\end{enumerate}
\end{itemize}
\end{proposition}
\begin{proof}
$(i)$\,
The proof of the proposition  is analogous to the one of   \cite[Proposition 4]{AngGol18a}. It uses Proposition \ref{signeigen} and a classical continuation argument based on the Riesz projection. 

$(ii)$ $(a)$\,
Observe that the operator $L^\alpha_{1,0}$ is the self-adjoint extension of the non-negative symmetric operator 
\begin{align*}
&L_{0}\bb v= L^\alpha_{1,0}\bb v,\quad \bb v\in\dom(L_{0}),\\
&\dom(L_0)= \Big\{\mathbf{v}\in
H^2(\Gamma): v_1(0)=\ldots=v_N(0)=0, \sum\limits_{j=1}^N
v_j'(0)=0  \Big\},
\end{align*}
with deficiency indices $n_\pm(L_{0})=1$  (see the proof of \cite[Theorem 3.12-$(iii)$]{AngGol18}). Hence $n(L^\alpha_{1,0})\leq 1$ by \cite[\S 14, Theorem 16]{Nai68}. Since   $\la L^\alpha_{1,0}\bs{\varphi}_{0,\omega}^\alpha,\bs{\varphi}_{0,\omega}^\alpha\ra_{L^2(\Gamma)}=-(p-1)\|\bs{\varphi}_{0,\omega}^\alpha\|^{p+1}_{L^{p+1}(\Gamma)}<0$, we get the equality $n(L^\alpha_{1,0})=1$.

$(b)$\, Consider the restriction $L^\alpha_{1,0}|_{L^2_{\eq}(\Gamma)}$ as the self-adjoint extension of the following symmetric operator 
\begin{align*}
&\tilde{L}_0\bb v= L^\alpha_{1,0}\bb v,\quad \bb v\in\dom(\tilde{L}_0),\\
&\dom(\tilde{L}_0)= \left \{\mathbf{v}\in\dom(H_\alpha)\cap L^2_{\eq}(\Gamma): v_1(b_0)=\ldots=v_N(b_0)=0 \right \},
\end{align*}
where $b_0=-\frac{2}{(p-1)\sqrt{m^2-\omega^2}}c_0,$ and $c_0$ is defined in \eqref{Phi_k}. 
Let us prove that $\tilde{L}_0$ is non-negative.
Observe  that every component of the vector $\mathbf{v}=(v_j)_{j=1}^N\in H^2(\Gamma)$ satisfies the identity
\begin{equation*}\label{identity_graph_k}
-v_j''+\omega v_j-p(\tilde{\varphi}_{0,j})^{p-1}v_j=
\frac{-1}{\tilde{\varphi}_{0,j}'}\frac{d}{dx}\left[(\tilde{\varphi}_{0,j}')^2\frac{d}{dx}\left(\frac{v_j}{\tilde{\varphi}_{0,j}'}\right)\right],\,\, x\in\mathbb{R}_+\setminus \{b_0\}.
\end{equation*}
Using the above equality and integrating by parts, we get for
 $\mathbf{v}\in \dom({\tilde L}_{0})$,
 \begin{equation*}\begin{split}
&\la\tilde{L}_0\mathbf{v},\mathbf{v}\ra_{L^2_{\eq}(\Gamma)}=N\Big( \int\limits_{0}^{b_0-} + \int\limits_{b_0+}^{+\infty}\Big)
(\tilde{\varphi}_{0,1} ')^2\left|\frac{d}{dx}\left(\frac{v_1}{\tilde{\varphi}_{0,1} '}\right)\right|^2dx+
N\left[-v_1'{\overline v}_1+|v_1|^2\frac{\tilde\varphi_{0,1}''}{\tilde\varphi_{0,1}'}\right]^{\infty}_{0}\\&+N\left[v_1'{\overline v}_1-|v_1|^2\frac{\tilde\varphi_{0,1}''}{\tilde\varphi_{0,1}'}\right]^{b_0+}_{b_0-}=N\Big( \int\limits_{0}^{b_0-} + \int\limits_{b_0+}^{+\infty}\Big)
(\tilde{\varphi}_{0,1} ')^2\left|\frac{d}{dx}\left(\frac{v_1}{\tilde{\varphi}_{0,1} '}\right)\right|^2dx\geq 0.
\end{split}
\end{equation*}

The deficiency indices of $\tilde{L}_0$ are $$n_\pm(\tilde{L}_0)=\dim(\dom(L^\alpha_{1,0}|_{L^2_{\eq}(\Gamma)}))-\dim(\dom(\tilde{L}_0)) = 1.$$ Therefore, $n(L^\alpha_{1,0}|_{L^2_{\eq}(\Gamma)})\leq 1$ by \cite[\S 14, Theorem 16]{Nai68}.
One also has  that  $\bs{\varphi}_{0,\omega}^\alpha\in L^2_{\eq}(\Gamma)$ and $\la L^\alpha_{1,0}\bs{\varphi}_{0,\omega}^\alpha,\bs{\varphi}_{0,\omega}^\alpha\ra_{L^2(\Gamma)}=-(p-1)\|\bs{\varphi}_{0,\omega}^\alpha\|^{p+1}_{L^{p+1}(\Gamma)}<0$. Hence  $n(L^\alpha_{1,0}|_{L^2_{\eq}(\Gamma)})=1$.
 \end{proof}
\begin{remark}
In \cite{AngGol18}, we considered the NLS equation  with the $\delta$-interaction on the star graph $\Gamma$. In particular, we proved \cite[Theorem 1.1]{AngGol18} on the orbital instability of the profile $\bb \Phi_{\alpha,\delta}=\bs{\varphi}_{0,\omega}^\alpha$ (where $m^2-\omega^2$ has to be substituted by $\omega$). Using the proof of  item $(ii)-(b)$, we may complete the result of Theorem 1.1 in \cite{AngGol18}. 
  Indeed, one may deduce analogously that $n(\bb L_{1,\alpha}|_{L^2_{\eq}(\Gamma)})=1$, where  the operator $\bb L_{1,\alpha}$  is defined in \cite[Subsection 3.1]{AngGol18}.  Using Proposition 3.17 and Proposition 3.19 $(ii)-2),3)$ in \cite{AngGol18}, we can affirm for $\alpha>0$  the following two results.
\begin{itemize} 
\item[$(i)$] Let $3<p<5$, then there exists $\omega_2>\frac{\alpha^2}{N^2}$ such that $e^{i\omega t}\bb \Phi_{\alpha,\delta}$ is orbitally unstable in $\mc E_{\eq}(\Gamma)$ and therefore in $\mc E(\Gamma)$ for $\omega\in(\frac{\alpha^2}{N^2}, \omega_2)$. Moreover,  $e^{i\omega t}\bb \Phi_{\alpha,\delta}$ is orbitally stable in $\mc E_{\eq}(\Gamma)$ for $\omega>\omega_2$.
\item[$(ii)$]   Let $p\geq 5$, then  $e^{i\omega t}\bb \Phi_{\alpha,\delta}$ is orbitally unstable in $\mc E_{\eq}(\Gamma)$, and therefore in $\mc E(\Gamma)$ for $\omega>\frac{\alpha^2}{N^2}$.
\end{itemize}  
   \end{remark}
   \begin{remark} Using the approach of \cite{Kai17} (see Theorem 3.1), one may show that in $L^2(\Gamma)$,
   $$n(L_{1,k}^\alpha)=\left\{\begin{array}{c}
   k+1,\,\text{for}\, \alpha<0,\\
   N-k, \,\text{for}\, \alpha>0,
   \end{array}\right.\qquad k\in\left\{0,\ldots,\left[\tfrac{N-1}{2}\right]\right\}.$$
   \end{remark}
\subsection{Slope condition}
Let  $k\in\left\{0,\ldots,\left[\tfrac{N-1}{2}\right]\right\}$. In this subsection we study the sign of $\partial_\omega Q(\bb \Phi_{k,\omega}^\alpha)$. 
By the definition of $Q$ and $\bb \Phi_{k,\omega}^\alpha=(\bs{\varphi}_{k,\omega}^\alpha, i\omega \bs{\varphi}_{k,\omega}^\alpha)$, we get  (see \cite[Proposition 5]{AngGol18a})
\begin{equation}\label{qu}
Q(\bb \Phi_{k,\omega}^\alpha)=-\omega\|\bs{\varphi}_{k,\omega}^\alpha\|^2_{L^2(\Gamma)}=Q_{k,1}(\omega)Q_{k,2}(\omega),
\end{equation}
where 
\begin{equation}\label{q1q2}
\begin{split}
&Q_{k,1}(\omega)=-2\omega \left(\frac{p+1}{2}\right)^{\tfrac 2{p-1}}\frac{(m^2-\omega^2)^{\tfrac {5-p}{2(p-1)}}}{p-1},\\
Q_{k,2}(\omega)=&\int\limits_{\tfrac{-\alpha}{(2k-N)\sqrt{m^2-\omega^2}}}^1k(1-t^2)^{\tfrac {3-p}{p-1}}dt+\int\limits_{\tfrac{\alpha}{(2k-N)\sqrt{m^2-\omega^2}}}^1(N-k)(1-t^2)^{\tfrac {3-p}{p-1}}dt.
\end{split}
\end{equation}
Using the above formulas for $Q(\bb \Phi_{k,\omega}^\alpha)$, we obtain the result.
\begin{lemma}\label{slope}
Let $1<p<5$.
\begin{itemize}
\item[$(i)$] $\partial_\omega Q(\bb \Phi_{k,\omega}^\alpha)>0$ for 
$$\alpha<0,\quad\text{and}\quad \omega\in\Big(-m,\frac{-m\sqrt{p-1}}{2}\Big)\cup\Big(\frac{m\sqrt{p-1}}{2}, m\Big).$$
\item[$(ii)$] $\partial_\omega Q(\bb \Phi_{k,\omega}^\alpha)<0$ for 
$$\alpha>0,\quad\text{and}\quad \omega\in\Big(\frac{-m\sqrt{p-1}}{2},0\Big)\cup\Big(0,\frac{m\sqrt{p-1}}{2}\Big).$$
\end{itemize}
\end{lemma}
\begin{proof}
By \eqref{qu}, $\partial_\omega Q(\bb \Phi_{k,\omega}^\alpha)=Q'_{k,1}(\omega)Q_{k,2}(\omega)+Q_{k,1}(\omega)Q'_{k,2}(\omega).$  From \eqref{q1q2} we get
\begin{equation*}
\begin{split}
&Q'_{k,1}(\omega)=\frac{2}{p-1}\left(\frac{p+1}{2}\right)^{\tfrac 2{p-1}}(m^2-\omega^2)^{\frac{7-3p}{2(p-1)}}\Big(\frac{4\omega^2}{p-1}-m^2\Big),\\
&Q'_{k,2}(\omega)=\Big(1-\frac{\alpha^2}{(2k-N)^2(m^2-\omega^2)}\Big)^{\tfrac{3-p}{p-1}}\frac{\alpha\omega}{(m^2-\omega^2)^{3/2}}.
\end{split}
\end{equation*}

Observe that $Q_{k,2}(\omega)>0$ for any $\omega$ and $\alpha$.
It is easily seen that  
$$\left\{\begin{array}{c}
Q'_{k,1}(\omega)>0,\\
Q_{k,2}(\omega)>0,\\
Q'_{k,2}(\omega)<0,\\
Q_{k,1}(\omega)<0
\end{array}\right.\quad\Longleftrightarrow\quad \left\{\begin{array}{c}\alpha<0,\\
\omega\in\left(\frac{m\sqrt{p-1}}{2}, m\right)\end{array}\right.$$  implies $\partial_\omega Q(\bb \Phi_{k,\omega}^\alpha)>0$.
Analogously  $$\left\{\begin{array}{c}
Q'_{k,1}(\omega)>0,\\
Q_{k,2}(\omega)>0,\\
Q'_{k,2}(\omega)>0,\\
Q_{k,1}(\omega)>0
\end{array}\right.\quad\Longleftrightarrow\quad \left\{\begin{array}{c}\alpha<0,\\
\omega\in\left(-m, \frac{-m\sqrt{p-1}}{2}\right)\end{array}\right.$$  yields $\partial_\omega Q(\bb \Phi_{k,\omega}^\alpha)>0$. Finally, $(i)$ holds.

To show $(ii)$, observe that 
$$\left\{\begin{array}{c}
Q'_{k,1}(\omega)<0,\\
Q_{k,2}(\omega)>0,\\
Q'_{k,2}(\omega)>0,\\
Q_{k,1}(\omega)<0
\end{array}\right. \quad\Longleftrightarrow\quad \left\{\begin{array}{c}\alpha>0,\\
\omega\in\left(0,\frac{m\sqrt{p-1}}{2}\right)\end{array}\right.$$ implies $\partial_\omega Q(\bb \Phi_{k,\omega}^\alpha)<0$, and  
$$\left\{\begin{array}{c}
Q'_{k,1}(\omega)<0,\\
Q_{k,2}(\omega)>0,\\
Q'_{k,2}(\omega)<0,\\
Q_{k,1}(\omega)>0
\end{array}\right.\quad\Longleftrightarrow\quad \left\{\begin{array}{c}\alpha>0,\\
\omega\in\left(\frac{-m\sqrt{p-1}}{2},0\right)\end{array}\right.$$ yields $\partial_\omega Q(\bb \Phi_{k,\omega}^\alpha)<0$. 
\end{proof}
Combining Lemma \ref{well_X_k}, Theorem \ref{thm_stabil}, Proposition \ref{spec}, Proposition \ref{spect_L_j},  Proposition \ref{grafoN2}, and Proposition \ref{n(L_k)}, we get stability/instability result.
\begin{theorem}\label{main}
Assume that  $1<p<5$ and $m>0$.
\begin{itemize}
\item[$(i)$] Let $k\in\left\{1,\ldots,\left[\tfrac{N-1}{2}\right]\right\}$ 
and $m^2-\omega^2>\tfrac{\alpha^2}{(N-2k)^2}$. 
\begin{itemize}
\item[$(a)$] For $\alpha>0$ and $\omega\in\left(\frac{-m\sqrt{p-1}}{2},0\right)\cup\left(0,\frac{m\sqrt{p-1}}{2}\right)$ the standing wave $e^{i\omega t}\bs{\varphi}_{k,\omega}^\alpha$ is orbitally unstable in $X_k$ and therefore in $X$.
\item[$(b)$] For $\alpha<0$ and $\omega\in\left(-m,\frac{-m\sqrt{p-1}}{2}\right)\cup\left(\frac{m\sqrt{p-1}}{2}, m\right)$ (if such $\omega$ exists) the standing wave $e^{i\omega t}\bs{\varphi}_{k,\omega}^\alpha$ is linearly unstable.
\end{itemize} 
\item[$(ii)$] Let $k=0$ and $m^2-\omega^2>\tfrac{\alpha^2}{N^2}$.
\begin{itemize}
\item[$(a)$] For $\alpha>0$ and $\omega\in\left(\frac{-m\sqrt{p-1}}{2},0\right)\cup\left(0,\frac{m\sqrt{p-1}}{2}\right)$ the standing wave $e^{i\omega t}\bs{\varphi}_{0,\omega}^\alpha$ is orbitally unstable in $X_{\eq}$ and therefore in $X$.
\item[$(b)$]  For $\alpha<0$ and $\omega\in\left(-m,\frac{-m\sqrt{p-1}}{2}\right)\cup\left(\frac{m\sqrt{p-1}}{2}, m\right)$ (if such $\omega$ exists)  the standing wave $e^{i\omega t}\bs{\varphi}_{0,\omega}^\alpha$ is orbitally stable in $X$.
\end{itemize}
\end{itemize}
\end{theorem}
\begin{proof}
Observe that, by Proposition \ref{spec}, Proposition \ref{spect_L_j}-$(vi)$, Proposition \ref{grafoN2}-$(iii)$, assumptions $(A3)$ and  $(A4)$  are satisfied.
\begin{itemize}
\item[$(i)$]

$(a)$\, By Proposition \ref{n(L_k)}-$(i)$, for  $\alpha>0$ we have  $n(L^\alpha_{1,k}|_{L^2_k(\Gamma)})=1$. If additionally  $\omega\in\left(\frac{-m\sqrt{p-1}}{2},0\right)\cup\left(0,\frac{m\sqrt{p-1}}{2}\right)$, then $\partial_\omega Q(\bb \Phi_{k,\omega}^\alpha)<0$ (see Lemma \ref{slope}-$(ii)$).

By Proposition \ref{spec}, \ref{grafoN2}-$(i)$ and  \ref{spect_L_j}-$(ii)$, we obtain $n(\mc L^\alpha_k|_{X_k})=1$.
Finally, by Theorem \ref{thm_stabil}-$(ii)$, we get orbital instability of  $e^{i\omega t}\bs{\varphi}_{k,\omega}^\alpha$
in $X_k$ since \eqref{Cauchy2} is well-posed in $X_k$ by Lemma \ref{well_X_k}. Consequently   $e^{i\omega t}\bs{\varphi}_{k,\omega}^\alpha$ is orbitally unstable in $X$.

$(b)$\, By Proposition \ref{n(L_k)}-$(i)$, for  $\alpha<0$ we have  $n(L^\alpha_{1,k}|_{L^2_k(\Gamma)})=2$ and analogously to the previous case  $n(\mc L^\alpha_k|_{X_k})=2$. If 
additionally $\omega\in\left(-m,\frac{-m\sqrt{p-1}}{2}\right)\cup\left(\frac{m\sqrt{p-1}}{2}, m\right)$, then $\partial_\omega Q(\bb \Phi_{k,\omega}^\alpha)>0$ (see Lemma \ref{slope}-$(i)$).
Finally, by Theorem \ref{thm_stabil}-$(i)$, we get linear instability of  $e^{i\omega t}\bs{\varphi}_{k,\omega}^\alpha$.

\item[$(ii)$] 

$(a)$\, By Proposition \ref{n(L_k)}-$(ii)$, for  $\alpha>0$ we have  $n(L^\alpha_{1,k}|_{L^2_{\eq}(\Gamma)})=1$, and therefore $n(\mc L^\alpha_k|_{X_{\eq}})=1$.  If additionally  $\omega\in\left(\frac{-m\sqrt{p-1}}{2},0\right)\cup\left(0,\frac{m\sqrt{p-1}}{2}\right)$, then $\partial_\omega Q(\bb \Phi_{0,\omega}^\alpha)<0$ (see Lemma \ref{slope}-$(ii)$). 

Finally, by Theorem \ref{thm_stabil}-$(ii)$, we get orbital instability of $e^{i\omega t}\bs{\varphi}_{0,\omega}^\alpha$
in $X_{\eq}$ since \eqref{Cauchy2} is well-posed in $X_{\eq}$ by Lemma \ref{well_X_k}. Consequently  $e^{i\omega t}\bs{\varphi}_{0,\omega}^\alpha$   is orbitally unstable in $X$.

$(b)$\,  By Proposition \ref{n(L_k)}-$(ii)$, for  $\alpha<0$ we have  $n(L^\alpha_{1,k})=1$ and consequently $n(\mc L^\alpha_k)=1$. If 
additionally $\omega\in\left(-m,\frac{-m\sqrt{p-1}}{2}\right)\cup\left(\frac{m\sqrt{p-1}}{2}, m\right)$, then $\partial_\omega Q(\bb \Phi_{0,\omega}^\alpha)>0$ (see Lemma \ref{slope}-$(i)$).
By Theorem \ref{thm_stabil}-$(i)$, we get orbital stability of  $e^{i\omega t}\bs{\varphi}_{0,\omega}^\alpha$
in $X$.
\end{itemize}
\end{proof}

\end{document}